\documentclass[11pt,oneside,english]{amsart}
\usepackage[LGR,T1]{fontenc}
\usepackage[latin9]{inputenc}
\usepackage{geometry}
\usepackage{amsthm}
\usepackage{amssymb}
\usepackage{setspace}
\usepackage{esint}
\makeatletter

\DeclareRobustCommand{\greektext}{%
  \fontencoding{LGR}\selectfont\def\encodingdefault{LGR}}
\DeclareRobustCommand{\textgreek}[1]{\leavevmode{\greektext #1}}
\DeclareFontEncoding{LGR}{}{}
\DeclareTextSymbol{\~}{LGR}{126}

\numberwithin{equation}{section}
\numberwithin{figure}{section}
\theoremstyle{plain}
\newtheorem{thm}{\protect\theoremname}
  \theoremstyle{definition}
  \newtheorem{defn}[thm]{\protect\definitionname}
  \theoremstyle{remark}
  \newtheorem{rem}[thm]{\protect\remarkname}
  \theoremstyle{plain}
  \newtheorem{lem}[thm]{\protect\lemmaname}
  \theoremstyle{plain}
  \newtheorem{prop}[thm]{\protect\propositionname}
  \theoremstyle{plain}
  \newtheorem{cor}[thm]{\protect\corollaryname}

\usepackage{bbm}
\usepackage{dsfont}
\allowdisplaybreaks[1]

\makeatother

\usepackage{babel}
  \providecommand{\corollaryname}{Corollary}
  \providecommand{\definitionname}{Definition}
  \providecommand{\lemmaname}{Lemma}
  \providecommand{\propositionname}{Proposition}
  \providecommand{\remarkname}{Remark}
\providecommand{\theoremname}{Theorem}

\begin{document}
\global\long\def\bbN{\mathbb{N}}

\global\long\def\bbQ{\mathbb{Q}}

\global\long\def\bbR{\mathbb{R}}

\global\long\def\bbZ{\mathbb{Z}}

\global\long\def\bbC{\mathbb{C}}

\global\long\def\eset{\emptyset}

\global\long\def\nto{\nrightarrow}

\global\long\def\re{\mathrm{Re\,}}

\global\long\def\im{\mathrm{Im\,}}

\global\long\def\limti{{\displaystyle \lim_{n\to\infty}}}

\global\long\def\sumnti{{\displaystyle \sum_{n=1}^{\infty}}}

\global\long\def\sumktn{{\displaystyle \sum_{k=1}^{n}}}

\global\long\def\E{{\bf E}}

\global\long\def\ind{\mathbbm{1}}

\global\long\def\O{O}

\global\long\def\Leb{\mathcal{\lambda}}

\title{Ergodic Properties of the Random Walk Adic \\
Transformation over the $\beta$-Transformation}

\author{Michael Bromberg\\
School of mathematical sciences, Tel Aviv University, Tel Aviv 69978,
Israel.}
\begin{abstract}
We define a random walk adic transformation associated to an aperiodic
random walk on $G=\bbZ^{k}\times\bbR^{D-k}$ driven by a $\beta$-transformation
and study its ergodic properties. In particular, this transformation
is conservative, ergodic, infinite measure preserving and we prove
that it is asymptotically distributionally stable and bounded rationally
ergodic. Related earlier work appears in \cite{AS} and \cite{ANSS}
for random walk adic transformations associated to an aperiodic random
walk driven by a subshift of finite type.%
\thanks{This research was partially supported by ISF grant 1599/13.%
}%
\thanks{This paper is part of a thesis submitted in partial fulfillment of
the requirements for the degree of Doctor of Philosophy (Mathematics)
at the Tel Aviv University%
} 
\end{abstract}
\maketitle

\section{Introduction}

Let $\left(X,\mathcal{B},m\right)$ be a $\sigma$-finite measure
space with infinite measure $m$ and let $T:X\rightarrow X$ be a
conservative, ergodic transformation preserving the measure $m$.
It is a consequence of Hopf's ratio ergodic theorem that for every
$f\in L^{1}\left(m\right)$, the normalized Birkhoff sums $\frac{1}{n}\sum_{k=0}^{n-1}f\circ T^{n}\left(x\right)$
tend to $0$ a.e. Moreover, (see \cite[Theroem 2.4.2]{Aa1}), for
every sequence of constants $a_{n}>0$, either $\liminf_{n\rightarrow\infty}\frac{1}{a_{n}}\sum_{k=0}^{n-1}f\circ T^{n}\left(x\right)=0$
a.e. for all non-negative $f\in L^{1}\left(m\right)$ or there exists
a subsequence $n_{k}$ such that $\lim_{k\rightarrow\infty}\frac{1}{a_{n_{k}}}\sum_{j=0}^{n_{k}-1}f\circ T^{j}\left(x\right)=\infty$
a.e for all non-negative $f\in L^{1}\left(m\right)$. It follows that
there is no sequence of constants $a_{n}$, such that $\sum_{k=0}^{n-1}f\circ T^{n}\propto a_{n}$.
$ $Nevertheless, for certain transformations, there are weaker types
of convergence for which $\frac{1}{a_{n}}\sum_{k=0}^{n-1}f\circ T^{n}$
converges. One such notion, which we proceed to define in the following
paragraph and is the subject of study in this paper, is that of distributional
stability (see \cite[3.6]{Aa1}). 

Recall that convergence in distribution of a sequence of random variables
$f_{n}$ to a random variable $f$ all taking values in some Polish
space $\mathcal{C}$, means that $E\left(g\circ f_{n}\right)\underset{n\rightarrow\infty}{\longrightarrow}E\left(g\circ f\right)$
for all bounded, continuous $g:\mathcal{C}\rightarrow\bbR$. Let $\left(X,\mathcal{B},m\right)$
be a $\sigma$-finite, infinite measure space, $f_{n}:X\rightarrow\left[0,\infty\right]$
be measurable and let $f\in\left[0,\infty\right]$ be a random variable
defined on some probability space. Let $\nu$ be some probability
measure absolutely continuous with respect to $m$. Then $\left\{ f_{n}\right\} $
may be viewed as a sequence of random variables defined on the probability
space $\left(X,\mathcal{B},\nu\right)$ and we write $f_{n}\overset{\nu}{\longrightarrow}f$
if $f_{n}$ converges in distribution to $f$. We say that $f_{n}$
converges strongly in distribution to $f$ and write $f_{n}\overset{\mathcal{L}\left(m\right)}{\longrightarrow}f$
if $f_{n}\overset{\nu}{\longrightarrow}f$ with respect to any probability
measure $\nu$, absolutely continuous with respect to $m$. Equivalently,
this means that $g\circ f_{n}\rightarrow E\left(g\left(f\right)\right)$
weak-$*$ in $L^{\infty}\left(m\right)$ for each bounded and continuous
$g:\left[0,\infty\right]\rightarrow\bbR$, i.e. $\int g\circ f_{n}\cdot p\, dm\rightarrow E\left(g\circ f\right)\int p\, dm$
for all $p\in L^{1}\left(m\right)$ (here and throughout this paper,
the space $\left[0,\infty\right]$ is the one point compactification
of $\left[0,\infty\right)$). 
\begin{defn}
A conservative, ergodic measure preserving transformation $\left(X,\mathcal{B},m,T\right)$
is distributionally stable if there is a sequence of constants $a_{n}>0$,
and a random variable $Y$ taking values in $\left(0,\infty\right)$,
such that 
\begin{equation}
\frac{1}{a_{n}}S_{n}\left(f\right)\overset{\mathcal{L}\left(m\right)}{\rightarrow}Ym\left(f\right)\label{eq: Distributional Stability}
\end{equation}
 for all $f\in L^{1}\left(m\right)$, $f\geq0$, where $S_{n}\left(f\right):=\sum_{k=0}^{n-1}f\circ T^{k}$
and $m\left(f\right):=\int_{X}f\, dm$. 
\end{defn}
Note that by Hopf's rational ergodic theorem if (\ref{eq: Distributional Stability})
holds for some $f\in L^{1}\left(m\right)$, $f\geq0$ then it holds
for all $f\in L^{1}\left(m\right)$. Moreover, if (\ref{eq: Distributional Stability})
holds, then the sequence $a_{n}$ is unique up to asymptotic equality
and is called the return sequence of $a_{n}$. In \cite{Aa2} distributional
stability was proved for pointwise dual ergodic transformations having
regularly varying return sequences with Mittag-Leffler distributions
appearing as limits (see also \cite{Aa2}). More recently, (see \cite{AS},
\cite{ADDS}) distributional stability was proved for certain transformations
with exponential chi-squared distributions appearing as limits. In
particular, it is proved in \cite{AS} that the random walk adic transformation
associated with an aperiodic random walk on $\bbZ^{k}\times\bbR^{D-k}$
driven by a subshift of finite type is distributionally stable with
exponential chi squared distribution (with $D$ degrees of freedom)
in the limit. One purpose of this paper to define a random walk adic
transformation associated with an aperiodic walk driven by the $\beta$-transformation,
and to prove that it is distributionally stable with chi squared exponential
distribution in the limit (see theorem \ref{thm:Main Theorem}).

Another notion which we study in this paper is that of bounded rational
ergodicity. 
\begin{defn}
A consrvative, ergodic, infinite measure preserving transformation
$\left(X,\mathcal{B},m,T\right)$ is called \textit{bounded rationally
ergodic }(see \cite{Aa3}) if there exists a measurable set $A\subseteq X$
with $0<m\left(A\right)<\infty$, such that there exists $M>0$, such
that for all $n\geq1$
\begin{equation}
\left\Vert \sum_{k=0}^{n-1}\ind_{A}\circ T^{k}\right\Vert {}_{L^{\infty}\left(A\right)}\leq M\int\limits _{A}\left(\sum_{k=0}^{n-1}\ind_{A}\circ T^{k}\right)dm.\label{eq:Bdd rat ergodicity}
\end{equation}
The rate of growth of the sequence $a_{n}=\frac{1}{m\left(A^{2}\right)}\int_{A}\left(\sum_{k=0}^{n-1}\ind_{A}\circ T^{k}\right)dm$
does not depend on a set $A$ satisfying (\ref{eq:Bdd rat ergodicity}).
Bounded rational ergodicity implies a kind of absolutely normalized
ergodic theorem stating that 
\[
\frac{S_{n}\left(f\right)}{a_{n}}\rightsquigarrow\int\limits _{X}fdm,\;\forall f\in L^{1}\left(m\right)
\]
where $f_{n}\rightsquigarrow f$ means that $\forall m_{l}\uparrow\infty$
$\exists n_{k}=m_{l_{k}}\uparrow\infty$ such that $\forall p_{j}=n_{k_{j}}$
such that $\forall p_{j}=n_{k_{j}}\uparrow\infty$, we have $\frac{1}{N}\sum_{j=1}^{N}f_{p_{j}}\underset{N\rightarrow\infty}{\longrightarrow}f$
a.e. In section (\ref{sec:Bounded-rational-ergodicity}) we prove
that the random walk adic transformation associated with an aperiodic
random walk driven by the $\beta$-transformation is bounded rationally
ergodic with $a_{n}\propto\frac{n}{\sqrt{\log n}}$. Bounded rational
ergodicity of random walk adic transformations associated to an aperiodic
random walk driven by a subshift of finite type are studied in \cite{ANSS}. 
\end{defn}

\section{\label{sec:adic-transformation}adic transformation associated with
a $\beta$-transformation}

\subsection{$\beta$-Transformations.}

In this section we give some preliminaries concerning $\beta$- transformations
and refer the reader to \cite{DK,Pa,Re,Bl} for proofs of all facts
stated herein. 

In what follows $\left[x\right]:=\min\left\{ n\in\bbZ:n\leq x\right\} $,
$\left(x\right)=x-\left[x\right]$, $\Leb$ is the Lebesgue measure
on $\left[0,1\right)$. 

The beta transformation is defined on $X=\left[0,1\right)$ by 
\[
T_{\beta}\left(x\right):=\beta x\, mod\,1\equiv\left(\beta x\right)
\]
It was proved by Rényi (see \cite{Re}) that there exists a unique,
ergodic, $T_{\beta}$-invariant measure $m$, equivalent to the Lebesgue
measure on $X$. Moreover, the invariant density, which we denote
by $h$, satisfies $1-\frac{1}{\beta}\leq h\left(x\right)\leq\frac{1}{1-\frac{1}{\beta}}$
for all $x\in X$. Every $x\in X$ has a $\beta$-expansion of the
form $x=\sum_{k=1}^{\infty}\frac{d_{k}\left(x\right)}{\beta^{k}}$
where $d_{k}\left(x\right):=\left[\beta T_{\beta}^{k-1}x\right]$.
Although $1$ is not in the domain of $T$, defining $T\left(1\right)=\beta-\left[\beta\right]$,
we can still consider the $\beta$-expansion of $1$ given by $1=\sum_{k=1}^{\infty}\frac{d_{k}\left(1\right)}{\beta^{k}}$,
where $d_{k}\left(1\right):=\left[\beta T_{\beta}^{k-1}1\right]$.
In what follows, we denote by $d\left(x,\beta\right)$ the sequence
of digits in the $\beta$-expansion of $x\in X\bigcup\left\{ 1\right\} $,
i.e $d\left(x,\beta\right)=\left(d_{1}\left(x\right),d_{2}\left(x\right),...\right)$
. Not every sequence of integers between $1$ and $\left[\beta\right]$
gives rise to a $\beta$-expansion of some $x\in X$. We say that
the sequence $\left(d_{1},d_{2},...\right)\in\left\{ 0,...,\left[\beta\right]\right\} ^{\bbZ}$
is $\beta$-admissible if it is the $\beta$-expansion of some $x\in X$,
i.e if $d_{k}=d_{k}\left(x\right)$ for some $x\in X$. The set of
$\beta$-admissible sequences, which we denote by $S_{\beta}$, is
a closed shift invariant subspace of $\left\{ 0,...,\left[\beta\right]\right\} ^{\bbZ}$.
This set may be linearly ordered by lexicographic order which we denote
by $\prec_{lex}$ in the obvious way. Namely, for distinct $\omega,\eta\in S_{\beta}$,
\[
\omega\prec_{lex}\eta
\]
 if there exists $n\in\bbN$ such that $\omega_{i}=\eta_{i}$ for
all $i<n$ and $\omega_{n}<\eta_{n}$. 

The map 
\begin{equation}
\psi\left(x\right)=\left(d_{1}\left(x\right),d_{2}\left(x\right)...\right)\label{eq: Conjugacy}
\end{equation}
 is one to one, onto, bi-measurable from $X$ to $S_{\beta}$ and
satisfies $\psi^{-1}\circ\sigma\circ\psi\left(x\right)=T_{\beta}\left(x\right)$
for all $x\in X$, where $\sigma$ is the left shift on $S_{\beta}$.
Thus, $\psi$ is an isomorphism between the systems $\left(X,\mathcal{B},m,T_{\beta}\right)$
and $\left(S_{\beta},\mathcal{C},\nu,\sigma\right)$ where $\mathcal{C}$
is the natural (Borel) $\sigma$-algebra on $\left\{ 0,...,\left[\beta\right]\right\} ^{\bbZ}$
restricted to $S_{\beta}$, and $\nu$ is the push forward of $m$
by $\psi$. If the $\beta$-expansion of $1$ is finite then $\left(S_{\beta},\sigma\right)$
is a subshift of finite type and as the case of subshifts of finite
types has been dealt with in \cite{AS}, we shall assume that $d\left(1,\beta\right)$
is not finite. In this case, the set of admissible sequences is identified
by the following theorem \cite{Pa}: 
\begin{thm}
If $d\left(1,\beta\right)$ is not eventually periodic than the sequence
$\omega=\left(d_{1},d_{2},...\right)$ is $\beta$-admissible if and
only if $\sigma^{n}\omega\prec_{lex}d\left(1,\beta\right)$ for all
$n\in\bbN$. 
\end{thm}
Thus, in the case that $d\left(1,\beta\right)$ is not eventually
periodic 
\[
S_{\beta}=\left\{ \omega\in\left\{ 0,...,\left[\beta\right]\right\} ^{\bbZ}:\sigma^{n}\omega\prec_{lex}d\left(1,\beta\right)\right\} .
\]

\begin{rem}
Henceforth, we assume that $d\left(1,\beta\right)$ is not eventually
periodic. 
\end{rem}
Let $\left[d_{1},...,d_{k}\right]:=\left\{ x\in X:x=\sum_{i=1}^{\infty}\frac{d_{i}\left(x\right)}{\beta^{k}},\, d_{i}=d_{i}\left(x\right),\, i=1,...,k\right\} $.
$\left[d_{1},...,d_{k}\right]$ is called a cylinder of rank $k$.
A cylinder $\left[d_{1},...,d_{k}\right]$ is called a full cylinder
of rank $k$ if 
\[
\Leb\left(T^{k}\left[d_{1},...,d_{k}\right]\right)=1
\]
 and non-full otherwise. All full cylinders $\Delta_{k}$ of rank
$k$ satisfy $\Leb\left(\Delta_{k}\right)=\frac{1}{\beta^{k}}$ (see
\cite{DK}) and therefore have equal Lebesgue measure. We state the
following lemma for future reference.
\begin{lem}
\label{lem: Full Intervals Lemma}

\cite{DK} Given any $k\in\bbN$, $X$ can be covered by disjoint
full intervals of rank $k$ or $k+1$. 
\end{lem}

\subsection{\label{sub:Adic-transformation}Adic transformation of the $\beta$-transformation. }

The purpose of this section is to define the adic transformation of
the $\beta$-transformation. Adic transformations appear in \cite{Ve},
where they are defined over Bratelli diagrams. We briefly describe
the construction. Let $\mathcal{S}_{k}=\left\{ 0,...,a_{k}\right\} $
be a sequence of finite alphabets and let $A_{k}:\mathcal{S}_{k}\times\mathcal{S}_{k+1}\rightarrow\left\{ 0,1\right\} $
be a sequence of transition matrices. Define $\varOmega:=\prod_{i=1}^{\infty}\mathcal{S}_{k}$
and 
\[
\Sigma:=\left\{ \omega\in\varOmega:A_{k}\left(\omega_{k},\omega_{k+1}\right)=1\right\} .
\]
The adic transformation over the Bratelli diagram $\left\{ \mathcal{S}_{k},A_{k}\right\} $
assigns to $\omega\in\Sigma$ the element of $\Sigma$ that succeeds
$\omega$ in the reverse lexicographic order (see definition (\ref{def:Reverse lexicographic order})
below). The adic transformation over the $\beta$-transformation will
be defined in a similar way, with the exception that the $\beta$-transformation
is not a Bratelli diagram since the set of allowable digits appearing
in the $n$th place of the $\beta$-expansion of a number $x\in X$
depends on the whole prefix and not only on the preceding digit in
the expansion. 
\begin{defn}
The tail relation of the $\beta$-transformation is the equivalence
relation on $X$ given by 
\begin{eqnarray*}
\mathcal{T}\left(T_{\beta}\right): & = & \left\{ \left(x,y\right)\in X\times X:\,\exists K\in\bbN\: such\: that\: d_{k}\left(x\right)=d_{k}\left(y\right)\:\forall k\geq K\right\} \\
 & = & \bigcup_{n\geq0}\left\{ \left(x,y\right):T_{\beta}^{n}x=T_{\beta}^{n}y\right\} 
\end{eqnarray*}

\end{defn}
\global\long\def\newmacroname{}

\begin{defn}
\label{def:Reverse lexicographic order}The reverse lexicographic
order on $X$ is the partial order $\prec_{rev}$ defined by $x\prec_{rev}y$
if and only if there exists $n\in\bbN$, such that $d_{k}\left(x\right)=d_{k}\left(y\right)$
for all $k>N$ and $d_{N}\left(x\right)<d_{n}\left(y\right)$. 
\end{defn}
Thus the equivalence sets of $\mathcal{T}\left(T_{\beta}\right)$
are linearly ordered by $\prec_{rev}$. 

$ $The adic transformation $\tau:X\rightarrow X$ of $T_{\beta}$
is that transformation which parametrizes the tail relation on $X$
(in the sense that $x\sim_{\mathcal{T}\left(X\right)}y$ if and only
if $y=\tau^{n}x$ for some $n\in\bbZ$) and assigns to each $x$ the
minimal $y$ that satisfies $y\succ_{rev}x$ . Thus $\tau$ is defined
by $\tau\left(x\right):=\min\left\{ y:\, d\left(y,\beta\right)\succ_{rev}d\left(x,\beta\right)\right\} $
where the minimum is taken with respect to $\prec_{rev}$. 

Our next objective is to identify the set on which $\tau$ is well
defined, and on which $\tau$ is invertible. To this purpose, we identify
the set of maximal point of $X$ with respect to the reverse lexicographic
order, and show that $\tau$ is well defined outside this set. 
\begin{prop}
\label{prop: Identification of adic transformation}Let 
\[
\Sigma_{max}:=\left\{ x:T_{\beta}^{n}x\geq1-\frac{1}{\beta},\,\forall n\in\bbN_{*}\right\} =\left\{ x:\,\forall n\in\bbN_{*}\, x\in T^{-n}\left[1-\beta,1\right)\right\} .
\]
Then $\tau$ is defined for all $x\in X\setminus\Sigma_{max}$ and
is not defined on $\Sigma_{max}$. Moreover, if for\\* $x\in X\setminus\Sigma_{max}$
\[
d\left(x,\beta\right)=\left(d_{1}\left(x\right),d_{2}\left(x\right),...\right)
\]
 is the sequence of digits in the $\beta$-expansion of $x$, then
\[
\tau\left(x\right)=\left(\left[0\right]_{n_{0}},d_{n_{0}+1}\left(x\right)+1,d_{n_{0}+2}\left(x\right),d_{n_{0}+3}\left(x\right),...\right)
\]
where $n_{0}:=\min\left\{ n\in\bbN_{*}:T_{\beta}^{n}<1-\frac{1}{\beta}\right\} $
and $\left[0\right]_{n}:=\underset{n\, times}{0,...,0}$. \end{prop}
\begin{proof}
Let $x\in X\setminus\Sigma_{max}$, $n_{0}:=\min\left\{ n:\, T_{\beta}^{n}x<1-\frac{1}{\beta}\right\} $
and let $d\left(x,\beta\right)=\left(d_{1}\left(x\right),d_{2}\left(x\right),...\right)$
be the $\beta$-expansion of $x$. Since $\psi^{-1}\circ\sigma\circ\psi\left(x\right)=T_{\beta}\left(x\right)$
where $\psi$ is as in (\ref{eq: Conjugacy}), it follows that the
$\beta$-expansion of $T_{\beta}^{n_{0}}x$ is $\left(d_{n_{0}+1}\left(x\right),d_{n_{0}+2}\left(x\right),...\right)$.
Now, note that if $y\in X$, $y+\frac{1}{\beta}<1$ then 
\[
d\left(y+\frac{1}{\beta},\beta\right)=\left(d_{1}\left(y\right)+1,d_{2}\left(y\right),d_{3}\left(y\right),...\right).
\]
This is seen as follows: the first digit of the $\beta$-expansion
of $y+\frac{1}{\beta}$ is 
\[
\left[\beta\left(y+\frac{1}{\beta}\right)\right]=\left[\beta y+1\right]=\left[\beta y\right]+1,
\]
while the remaining digits are formed by the expansion of $y+\frac{1}{\beta}-\frac{\left[\beta y\right]+1}{\beta}=y-\frac{\left[\beta y\right]}{\beta}$,
which in turn, coincide with the remaining digits in the expansion
of $y$. Since $T_{\beta}^{n_{0}}x+\frac{1}{\beta}<1$, it follows
that 
\[
d\left(T_{\beta}^{n_{0}}x+\frac{1}{\beta},\beta\right)=\left(d_{n_{0}+1}\left(x\right)+1,d_{n_{0}+2}\left(x\right),d_{n_{0}+3}\left(x\right),...\right)
\]
 is an admissible sequence and therefore, $\omega=\left(\left[0\right]_{n_{0}},d_{n_{0}+1}\left(x\right)+1,d_{n_{0}+2}\left(x\right),d_{n_{0}+3}\left(x\right),...\right)$
is also admissible ($\sigma^{n}\omega\prec_{lex}d\left(1,\beta\right)$
for $n<n_{0}$ since the first digit in $d\left(1,\beta\right)$ is
$\left[\beta\right]>0$, and $\sigma^{n}\omega\prec_{lex}d\left(1,\beta\right)$
for $n\geq n_{0}$ since $\sigma^{n_{0}}\omega$ is admissible). Let
$y\in X$ be such that 
\[
d\left(y,\beta\right)=\omega.
\]
Obviously $x\prec_{rev}y$, and by definition of the reverse lexicographic
order there must be only finite number of elements that lie strictly
between $x$ and $y$. This implies that the set 
\[
\left\{ y:d\left(y,\beta\right)\succ_{rev}d\left(x,\beta\right)\right\} 
\]
 has a minimal element and therefore $\tau\left(x\right)$ is defined. 

We show that $\tau\left(x\right)=y$. If not, then there must exist
$z\in X$ , such that $x\prec_{rev}z\prec_{rev}\omega$. This implies
that there exists $k\leq n_{0}$ such that $d_{k}\left(x\right)<d_{k}\left(y\right)$
and $d_{n}\left(x\right)=d_{n}\left(y\right)$ for all $n>k$. This
implies that 
\[
\sigma^{k-1}\left(d\left(y,\beta\right)\right)\succeq_{lex}\left(d_{k}\left(x\right)+1,d_{k+1}\left(x\right),d_{k+2}\left(x\right),...\right),
\]
which in this case means that $\left(d_{k}\left(x\right)+1,d_{k+1}\left(x\right),d_{k+2}\left(x\right),...\right)$
is an admissible sequence. Therefore, $\frac{d_{k}\left(x\right)+1}{\beta}+\sum_{i=1}^{\infty}\frac{d_{k+1}\left(x\right)}{\beta^{i+1}}<1$
and it follows that 
\[
T_{\beta}^{k-1}x=\sum_{i=0}^{\infty}\frac{d_{k+i}\left(x\right)}{\beta^{i+1}}<1-\frac{1}{\beta}
\]
which contradicts the definition of $n_{0}$. 

Note that in particular, the last argument shows that if $x\prec_{lex}z$,
then there must be some $k\in\bbN_{*}$, such that $T_{\beta}^{k}x<1-\frac{1}{\beta}$
and therefore, $x\notin\Sigma_{max}$. This completes the proof.\end{proof}
\begin{cor}
There exists a measurable, $\tau$ invariant set $\hat{X}\subseteq X$
with $\Leb\left(\hat{X}\right)=1$ restricted to which $\tau$ is
invertible. \end{cor}
\begin{proof}
Using proposition \ref{prop: Identification of adic transformation}
inductively, we conclude that $\tau^{n}x$ is defined for all $n\in\bbN$,
if and only if $T_{\beta}^{n}x+\frac{1}{\beta}<1$ for infinitely
many $n\in\bbN$. Letting $\Sigma_{max}$ be as in proposition \ref{prop: Identification of adic transformation},
we have the equality 
\[
\left\{ x\in X:\, T_{\beta}^{n}x+\frac{1}{\beta}<1\: for\: finitely\: many\: n\in\bbN_{*}\right\} =\bigcup_{k\in\bbN_{*}}T_{\beta}^{-k}\Sigma_{max}.
\]
Thus, all powers of $\tau$$ $are well-defined for all $x\in Y:=X\setminus\bigcup_{k\in N_{*}}T^{-k}\Sigma$
and there is some power of $\tau$ which is undefined for $x\notin Y$.
It follows that $Y$ is $\tau$-invariant. 

We show that $Y$ has full Lebesgue measure. Since $T_{\beta}$ is
ergodic and $m\left(\left[1-\beta,1\right)\right)<1$, it follows
by Birkhoff's ergodic theorem that there exists $\rho<1$ such that
for $m$ almost every $x$, and large enough $n$, $\frac{1}{n}\sum_{k=0}^{n-1}\ind_{\left[0,1-\beta\right)}<\rho$.
This shows that $m\left(\Sigma\right)=0$ and therefore, $m\left(\bigcup_{k\in\bbN_{*}}T_{\beta}^{-k}\Sigma\right)=0$.
By equivalence of the Lebesgue measure and the measure $m$, it follows
that $\Leb\left(Y\right)=1$. 

Setting $\hat{X}:=Y\setminus\bigcup_{k=0}^{\infty}T_{\beta}^{-k}\left(\bar{0}\right)$,
where $\bar{0}=\left(0,0,....\right)$ we obtain a set of full Lebesgue
measure, invariant under $\tau$ (invariance is seen by the fact that
$\tau$ parametrizes the tail relation, i.e $x$ and $\tau x$ must
be in the same equivalence class). Since it is clear from the definition
of $\tau$ that it is an injective map, to prove invertibility, it
suffices to show that $\tau:\hat{X}\rightarrow\hat{X}$ is onto. Let
$x\in\hat{X}$ and let $d\left(x,\beta\right)=\left(d_{1}\left(x\right),d_{2}\left(x\right),...\right)$
be its expansion. Let $n_{0}:=min\left\{ k:d_{k}\left(x\right)>0\right\} $.
Then $\omega=\left(d_{1}\left(x\right),...,d_{k}\left(x\right)-1,d_{k+1}\left(x\right),...\right)$
is an admissible sequence if $y\in X$ has expansion $\omega$, then
$y\prec_{rev}x$ and there are finitely many elements between $y$
and $x$ ordered by the reverse lexicographic order. It follows that
$\tau^{k}x=y$ for some $k\in\bbN$, and that $y\in\hat{X}$. Therefore,
$\tau:\hat{X}\rightarrow\hat{X}$ is onto and the proof is complete. \end{proof}
\begin{prop}
$\tau:\hat{X}\rightarrow\hat{X}$ preserves the Lebesgue measure. \end{prop}
\begin{proof}
Let $x\in\hat{X}$. Since the set of all cylinders $\left\{ \Delta_{n}\right\} _{n=1}^{\infty}$
generates $\mathcal{B}$, it suffices to prove that there exists a
cylinder $\Delta_{n}$ such that $x\in\Delta_{n}$, and $\Leb\left(\tau\left(\Delta_{n}\right)\right)=\Leb\left(\Delta_{n}\right)$.
Let 
\[
n_{0}=\min\left\{ T^{n}x+\frac{1}{\beta}<1\right\} .
\]
Then by proposition \ref{prop: Identification of adic transformation}
\[
\tau x=\left(\left[0\right]_{n_{0}},d_{n_{0}+1}\left(x\right)+1,d_{n_{0}+2}\left(x\right),d_{n_{0}+3}\left(x\right),...\right).
\]
Fix $n_{1}>n_{0}+1$. By lemma \ref{lem: Full Intervals Lemma}, there
exist two cylinders $A$, $B$ of rank $n_{1}$ or $n_{1}+1$ such
that $x\in A$, $\tau x\in B$. If the ranks are different, by concatenating
the last symbol of the longer cylinder to the shorter cylinder we
obtain two full cylinders of equal rank. Doing this will not change
the fact that $x\in A$ and $\tau x\in B$, because the above formula
for $\tau x$ shows that the digits in the expansions of $x$ and
$\tau x$ coincide for the index $n_{1}+1$. Therefore, without loss
of generality, we may assume that both $A$ and $B$ are full cylinders
of rank $\tilde{n}>n_{0}+1$. It follows that 
\[
A=\left(d_{1}\left(x\right),d_{2}\left(x\right),...,d_{\tilde{n}}\left(x\right)\right)
\]
 and 
\[
B=\left(\left[0\right]_{n_{0}},d_{n_{0}+1}\left(x\right)+1,d_{n_{0}+2}\left(x\right),...,d_{\tilde{n}}\left(x\right)\right).
\]
Therefore, proposition \ref{prop: Identification of adic transformation}
shows that $y\in A\implies\tau y\in B$ and it is easy to see by definition
of the reverse lexicographic order that $\tau^{-1}\left(B\right)=A$.
Thus, $\tau\left(A\right)=B$ and since full intervals of equal rank
have same Lebesgue measure the claim follows.
\end{proof}

\section{\label{sec:Random-Walk-Adic}Random Walk Adic Transformation Associated
with an Aperiodic Random Walk for the Beta Transformation. }

Let $G=\bbR$ or $G=\bbZ$ and let $\varphi:X\rightarrow G$. The
random walk over the $\beta$-transformation generated by $f$ is
the skew product 
\[
\left(X\times G,\mathcal{B}\left(X\right)\times\mathcal{B}\left(G\right),\tilde{m},\sigma_{\varphi}\right)
\]
 where $\tilde{m}:=m\times dy$, $dy$ is the Haar measure on $G$
and $\sigma_{\varphi}\left(x,y\right)=\left(T_{\beta}x,y+\varphi\left(x\right)\right)$.
In what follows, Birkhoff sums of the form $\sum_{k=0}^{n-1}\varphi\left(T_{\beta}^{k}x\right)$
will be denoted by $\varphi_{n}\left(x\right)$. 

Similarly to how the adic transformation $\tau$ parametrizes the
tail relation of $T_{\beta}$, the random walk adic transformation
associated to $\sigma_{\varphi}$ is the (unique) skew product over
$\left(\hat{X},\mathcal{B}\bigcap\hat{X},\Leb\right)$, which parametrizes
the tail relation of $\sigma_{\varphi}$. To identify this note that
the tail relation of $\sigma_{\varphi}$ is given by 
\[
\mathcal{T}\left(\sigma_{\varphi}\right)=\left\{ \left(x,y\right)\times\left(x',y'\right):\left(x,x'\right)\in\mathcal{T}\left(T_{\beta}\right),\,\exists n_{0}\forall n>n_{0}\ y+\varphi_{n}\left(x\right)=y'+\varphi_{n}\left(x'\right)\right\} .
\]
Now let $\left(x,y\right)\times\left(x',y'\right)\in\mathcal{T}\left(\sigma_{\varphi}\right)$.
Since $\left(x,x'\right)\in\mathcal{T}\left(T_{\beta}\right)$, it
follows that there exists $n$ such that $T_{\beta}^{n-1}\left(x\right)=T_{\beta}^{n-1}\left(y\right)$.
Therefore, 
\[
y+\varphi_{k}\left(x\right)=y'+\varphi_{k}\left(x'\right)
\]
 for all $k$ greater than some $K\in\bbN$, if and only if $y+\varphi_{n}\left(x\right)=y'+\varphi_{n}\left(x'\right)$.
It follows that 
\[
\left(x,y\right)\times\left(x',y'\right)\in\mathcal{T}_{\sigma_{\varphi}}
\]
 if and only if 
\[
\left(x,x'\right)\in\mathcal{T}\left(T_{\beta}\right)
\]
 and 
\[
y'=y+\psi\left(x,x'\right),
\]
where $\psi\left(x,x'\right)=\sum_{k=0}^{\infty}\varphi\left(T_{\beta}^{k}x\right)-\varphi_{k}\left(T_{\beta}^{k}x'\right)$.
It follows that for $\left(x,y\right)\in\hat{X}\times\bbR^{d}$, 
\[
\left(\tau x,y+\phi\left(x\right)\right)\sim_{\mathcal{T}\left(\sigma_{\varphi}\right)}\left(x,y\right)
\]
 if and only if $\phi\left(x\right):=\psi\left(x,\tau x\right)$.
Thus, we define the random walk adic transformation as follows. 
\begin{defn}
The random walk adic transformation associated to $\sigma_{\varphi}$
is the skew product 
\[
\left(\hat{X}\times G,\left(\mathcal{B}\bigcap\hat{X}\right)\times\mathcal{B}\left(G\right),\mu,\tau_{\varphi}\right),
\]
where $\mbox{\ensuremath{\mu}}=\Leb\times dy$, $\lambda$ is the
Lebesgue measure on $X$ restricted to $\hat{X}$, $dy$ is the Haar
measure on $G$ and 
\[
\tau_{\varphi}\left(x,y\right)=\left(\tau x,y+\phi\left(x\right)\right)
\]
where 
\[
\phi\left(x\right):=\psi\left(x,\tau x\right)=\sum_{k=0}^{\infty}\varphi\left(T^{k}x\right)-\varphi\left(T^{k}\left(\tau x\right)\right).
\]

\end{defn}
Note that since $\tau$ is invertible on $\hat{X}$, $\tau_{\varphi}$
is invertible and by the arguments above, for $\left(x,y\right),\left(x',y'\right)\in\hat{X}\times\bbR^{d}$,
$\left(x,y\right)\sim_{\mathcal{T}\left(\sigma_{f}\right)}\left(x',y'\right)$
if and only if $\tau_{\varphi}^{n}\left(x,y\right)=\left(x',y'\right)$
for some $n\in\bbZ$.

Denote by $\hat{G}$ the dual group of $G$.
\begin{defn}
A measurable function $\varphi:X\rightarrow G$ is aperiodic if the
only solutions the equation $\gamma\circ\varphi=\frac{\lambda g}{g\circ T}$
$m$-a.e, with $\gamma\in\hat{G}$, $\left|\lambda\right|=1$ and
a measurable $g:X\rightarrow\mathbb{S}^{1}$ are $\gamma\equiv1$,
$\lambda=1$ and $g$ is an a.e constant function. 
\end{defn}
We say that the random walk over the $\beta$-transformation is aperiodic,
if it is generated by an aperiodic function $\varphi$. Aperiodicity
is crucial for proving exactness and local limit theorems for the
skew product $\left(X\times\bbR,\mathcal{B}\left(X\right)\times\mathcal{B}\left(G\right),\tilde{m},\sigma_{\varphi}\right)$.$ $
In order for these to hold, in addition to aperiodicity, we must make
further regularity assumptions on the function $\varphi$, namely
we need to restrict $\varphi$ to a Banach space, on which the associated
transfer operator (also known as the Ruelle-Frobenius-Perron operator)
acts quasi-compactly. This is the goal of the following section. All
relevant definitions are provided therein.

\section{\label{sec:Assumptions-on-the-observable}Assumptions on the observable
$\varphi$ and implications}

The results of this section appear in \cite{ADSZ} where they are
proved in a more general context of piecewise monotonic, expanding
maps of the interval. We list the results relevant to our case. 

For an interval $A\subseteq X$, and $f:A\rightarrow\bbR$, define
the variation of $f$ on $A$ to be $var_{f}\left(A\right):=\sup\sum_{i}\left|f\left(x_{i}\right)-f\left(x_{i-1}\right)\right|$
where the supremum is taken over all finite partitions 
\[
x_{1}<x_{2}<...<x_{n}
\]
 of $A$. For$f\in L^{1}\left(m\right)$ set 
\[
\bigvee_{A}f:=\inf\left\{ var_{f*}\left(A\right):f^{*}=f\ a.e\right\} .
\]
For $f\in L^{\infty}\left(m\right)$, define 
\[
\left\Vert f\right\Vert _{BV}:=\left\Vert f\right\Vert _{\infty}+\bigvee_{X}f
\]
 and let 
\[
BV:=\left\{ f\in L^{\infty}\left(m\right):\left\Vert f\right\Vert _{BV}<\infty\right\} 
\]
The space $BV$ endowed with the norm $\left\Vert \cdot\right\Vert _{BV}$
is a Banach space. 

We will also be interested in functions of bounded variation on each
element of the natural partition of the unit interval for the $\beta$-transformation.
This partition corresponds to the partition $\left\{ \left[1\right],...,\mbox{\ensuremath{\left[\beta\right]}}\right\} $
of the associated $\beta$-shift and is given by 
\[
\alpha=\left\{ \left[0,\frac{1}{\beta}\right),\left[\frac{1}{\beta},\frac{2}{\beta}\right)...,\left[\frac{\left[\beta\right]}{\beta},1\right)\right\} .
\]
We say that $\varphi:X\rightarrow\bbR$ is locally of bounded variation
on $\alpha$, if 
\[
C_{\varphi,\alpha}:=\sup_{A\in\alpha}\bigvee_{A}\varphi<\infty.
\]
Note that since $\alpha$ is a finite partition, $C_{\varphi,\alpha}<\infty$
implies that $\varphi$ is bounded. 

Recall that for a non-singular dynamical system $\left(Y,\mathcal{C},\mu,T\right)$
the transfer operator is an operator $\hat{T}:L^{1}\left(\mu\right)\rightarrow L^{1}\left(\mu\right)$,
uniquely defined by the equality 
\[
\int\limits _{X}f\circ\hat{T}\cdot g\, d\mu=\int\limits _{X}f\cdot g\circ T\, d\mu
\]
for every $f\in L^{1}\left(m\right)$, $g\in L^{\infty}\left(m\right)$.
Let $\hat{T}_{\beta}$ be the transfer operator of $\left(X,\mathcal{B},m,T_{\beta}\right)$. 

In what follows we also need the transfer operator $\mathcal{L}:L^{1}\left(m\right)\rightarrow L^{1}\left(m\right)$,
defined by 
\[
\left(\mathcal{L}f\right)\left(x\right)=\sum_{y\in T_{\beta}^{-1}\left(x\right)}f\left(x\right).
\]
Note that $\mathcal{L}f\left(x\right)$ is finite for almost every
$x\in X$, and $\mathcal{L}f=\mathcal{L}\tilde{f}\:\mod m$ if $f=\tilde{f}\:\mod m$.
The operator $\mathcal{L}$ is also referred to as the transfer operator
(or the Ruelle-Frobenius-Perron operator) and may be used to obtain
the $T_{\beta}$ invariant density $h$ (see \cite{Wa}). We have
that$\beta$ is an eigenvalue of $\mathcal{L}$ corresponding to the
function $h$, i.e $\mathcal{L}h=\beta h$ and the operator $\mathcal{L}$
and $\hat{T}$ are related by (see \cite[Lemma 11]{Wa}) 
\[
\mathcal{L}\left(f\right)=\beta h\cdot\hat{T}\left(\frac{f}{h}\right)\:\forall f\in L^{1}\left(m\right).
\]

\begin{defn}
An operator $G$ on a Banach space $B$ is called quasi-compact with
$s$ dominating simple eigenvalues if \end{defn}
\begin{enumerate}
\item There exist $G$-invariant spaces $F$ and $H$ such that $F$ is
an $s$ dimensional space and $B=F\oplus H$. 
\item $G$ is diagonizable when restricted to $F$ with all eigenvalues
having modulus equal to the spectral radius of $G$, denoted by $\rho\left(G\right)$.
\item When restricted to $H$, the spectral radius of $G$ is strictly less
than $\rho\left(G\right)$.\end{enumerate}
\begin{defn}
The fact that $\hat{T}_{\beta}$ is a quasi-compact operator on $BV$
with one simple dominating eigenvalue $1$ and corresponding eigenspace
of constant functions is proved in \cite{ADSZ}. Thus, $\hat{T}_{\beta}$
has the form $\hat{T}_{\beta}=m\left(f\right)\ind+Q$ where the spectral
radius of $Q:\mathcal{B}\rightarrow\mathcal{B}$ satisfies $\rho\left(Q\right)<1$
and $m\circ Q=Q\ind=0$. 

The characteristic function operator associated to a measurable function
$\varphi:X\rightarrow\bbR$ is a family of operators $P\left(t\right):L^{1}\left(m\right)\rightarrow L^{1}\left(m\right)$,
$t\in\bbR$ defined by 
\[
P\left(t\right)f=\hat{T}_{\beta}\left(e^{it\varphi}f\right).
\]

Let $\varphi:X\rightarrow\bbR$ be such that $C_{\varphi,\alpha}<\infty$.
Then for $t\in\bbR$, $P\left(t\right):BV\rightarrow BV$ is quasi-compact
and $P\left(t\right)$ is twice continuously differentiable as a function
from $\bbR$ to $Hom\left(BV,BV\right)$. It follows from operator
perturbation theory (see \cite{ADSZ}) that there exists a $\delta$
neighborhood of $0$, such that for $\left|t\right|<\delta$, $P\left(t\right)$
is quasi-compact with a simple dominating eigenvalue $\lambda\left(t\right)$,
where $\lambda\left(t\right)$ has Taylor's expansion at $0$ of the
form 
\[
\lambda\left(t\right)=1+im\left(\varphi\right)-\sigma^{2}t^{2}+o\left(t^{2}\right),\ \sigma\geq0.
\]
Moreover, $\sigma^{2}=\lim_{n\rightarrow\infty}\frac{1}{n}Var_{m}\left(\phi_{n}\right)$
(here $Var\left(\varphi_{n}\right)$ denotes the variance of the sum
$\varphi_{n}$) and $\sigma^{2}=0$ if and only if $\varphi$ is a
coboundary, i.e of the form $\varphi=f\circ T_{\beta}-f$ for some
measurable $f:X\rightarrow\bbR$. As a consequence of that, exactness,
conditional central limit theorems and conditional local theorems
for the skew product $\left(X\times G,\mathcal{B}\left(X\right)\times\mathcal{B}\left(G\right),\tilde{m},\sigma_{\varphi}\right)$
can be obtained (see \cite{ADSZ}, \cite{AD}, \cite{HH}). We list
these results. 

Through the rest of this paper we assume that $\varphi:X\rightarrow G$,
$G=\bbR$ or $G=\bbZ$, $C_{\varphi,\alpha}<\infty$, $\varphi$ is
aperiodic and $\lim_{n\rightarrow\infty}\frac{1}{n}Var_{m}\left(\varphi_{n}\right)=\sigma^{2}>0$.

Recall that a non-singular transformation on a standard probability
space $\left(Y,\mathcal{C},\mu,T\right)$ is exact if the tail $\sigma-$field
of $T$ defined by $\mathcal{T}\left(T\right):=\bigcap\limits _{n=1}^{\infty}T^{-n}\mathcal{C}$
is trivial, i.e $\mathcal{T}\left(T\right)=\left\{ \emptyset,Y\right\} $. \end{defn}
\begin{thm}
\cite[Theorem 7]{ADSZ} If $\varphi:X\rightarrow G$ where $G=\bbR$
or $G=\bbZ$ is aperiodic and $C_{\varphi,\alpha}<\infty$, then the
skew product $\left(X\times G,\mathcal{B}\left(X\right)\times\mathcal{B}\left(G\right),\tilde{m},\sigma_{\varphi}\right)$
is exact.\end{thm}
\begin{cor}
If $\varphi:X\rightarrow G$ is aperiodic and $C_{\varphi,\alpha}<\infty$
then the random walk adic transformation $\left(\hat{X}\times G,\left(\mathcal{B}\bigcap\hat{X}\right)\times\mathcal{B}\left(G\right),\mu,\tau_{\varphi}\right)$
is conservative and ergodic. $ $\end{cor}
\begin{proof}
Ergodicity follows from exactness of the skew product $\left(X\times G,\mathcal{B}\left(X\right)\times\mathcal{B}\left(G\right),\tilde{m},\sigma_{\varphi}\right)$.
Indeed, since $\tau_{\varphi}$ parametrizes the tail relation of
$\left(X\times G,\mathcal{B}\left(X\right)\times\mathcal{B}\left(G\right),\tilde{m},\sigma_{\varphi}\right)$
any $\tau_{\varphi}$ invariant subset must be in the tail $\sigma$-field
of $\sigma_{\varphi}$. Conservativity follows since $\tau_{\varphi}$
is invertible and ergodic (see \cite[Proposition 1.2.1]{Aa1}). \end{proof}
\begin{thm}
\cite[Theorem 9(1)]{ADSZ}(CLT) For an interval $I\subseteq\bbR$,
\[
\hat{T}_{\beta}^{n}\left(\ind_{\left\{ \frac{\bar{\varphi}_{n}}{\sigma\sqrt{n}}\in I\right\} }\right)\left(x\right)\underset{n\rightarrow\infty}{\longrightarrow}\frac{1}{\sqrt{2\pi}}\int\limits _{I}e^{-\frac{t^{2}}{2}}dt,
\]
uniformly in $x\in X$. In particular 
\[
m\left(\ind_{\left\{ \frac{\varphi_{n}}{\sqrt{n}}\in I\right\} }\right)\longrightarrow\frac{1}{\sqrt{2\pi}}\int\limits _{I}e^{-\frac{t^{2}}{2}}dt
\]
. \end{thm}
\begin{rem}
Aperiodicity of $\varphi$ is not required for the CLT. 
\end{rem}
\global\long\def\newmacroname{}

\begin{thm}
\cite[Theorem 9(2)]{ADSZ}(LLT - Discrete version) Assume that $\varphi:X\rightarrow\bbZ$
is aperiodic. Then 
\[
\sigma\sqrt{n}\hat{T}_{\beta}^{n}\left(\ind_{\left\{ \varphi_{n}=k_{n}\right\} }\right)\left(x\right)\underset{n\rightarrow\infty}{\longrightarrow}\frac{1}{\sqrt{2\pi}}e^{-\frac{t^{2}}{2}},\, k_{n}\in\bbZ,\,\frac{k_{n}-nE_{m}\left(\varphi\right)}{\sigma\sqrt{n}}\rightarrow t
\]
 uniformly in $x\in X$, $t\in K$, for all $K\subseteq\bbR$ compact. 
\end{thm}
\global\long\def\newmacroname{}

\begin{thm}
\cite[Theorem 9(3)]{ADSZ}(LLT- Continuous version) Assume that $\varphi:X\rightarrow\bbR$
is aperiodic and $I$ is a bounded interval. Then 
\[
\sigma\sqrt{n}\hat{T}_{\beta}^{n}\left(\ind_{\left\{ \varphi_{n}\in k_{n}+I\right\} }\right)\left(x\right)\underset{n\rightarrow\infty}{\longrightarrow}\frac{1}{\sqrt{2\pi}}e^{-\frac{t^{2}}{2}}\left(x\right),\, k_{n}\in\bbR,\,\frac{k_{n}-nE_{m}\left(\varphi\right)}{\sigma\sqrt{n}}\rightarrow t
\]
 uniformly in $x\in X$, $t\in K$, for all $K\subseteq\bbR$ compact.
$ $\end{thm}
\begin{rem}
Uniformity in $t$ in the above theorems should be interpreted as
follows: Let $K\subseteq\bbR$ be compact and assume that for every
$t\in K$ we have a sequence $k_{n}\left(t\right)$ such that $\frac{k_{n}\left(t\right)-nE_{m}\left(\varphi\right)}{\sigma\sqrt{n}}$
converges to $t$ uniformly as $n\rightarrow\infty$, then 
\[
\sigma\sqrt{n}\hat{T}_{\beta}^{n}\left(\ind_{\left\{ \varphi_{n}\in k_{n}\left(t\right)+I\right\} }\right)\left(x\right)\underset{n\rightarrow\infty}{\longrightarrow}\frac{1}{\sqrt{2\pi}}e^{-\frac{t^{2}}{2}}
\]
uniformly in $x\in X$, $t\in K$. 

The following theorems are a version of the two previous ones that
instead of giving actual limits provide an upper bound for $\hat{T}_{\beta}^{n}\left(\ind_{\left\{ \varphi_{n}=k\right\} }\right)$
and $\hat{T}_{\beta}^{n}\left(\ind_{\left\{ \varphi_{n}\in I+y\right\} }\right)$
for all $k\in\bbZ$, $y\in\bbR$. The proof is essentially the same
as the proofs of the LLT theorem. \end{rem}
\begin{thm}
\label{thm:Upper bound LLT}(Discrete version) Assume that $\varphi:X\rightarrow\bbZ$
is aperiodic. Then there exists a constant $C$ such that 
\[
\hat{T}_{\beta}^{n}\left(\ind_{\left\{ \varphi_{n}=k\right\} }\right)\left(x\right)\leq\frac{C}{\sqrt{n}}
\]
for all $k\in K,$ $x\in X$. 
\end{thm}
\global\long\def\newmacroname{}

\begin{thm}
\label{thm:Upper bound LLT - continuous}(Continuous version) Assume
that $\varphi:X\rightarrow\bbR$ is aperiodic and $I\subseteq\bbR$
is a bounded interval. Then there exists a constant $C$ such that
\[
\hat{T}_{\beta}^{n}\left(\ind_{\left\{ \varphi_{n}\in I+y\right\} }\right)\left(x\right)\leq\frac{C}{\sqrt{n}}
\]
for all $y\in\bbR$, $x\in X$. 
\end{thm}

\section{Asymptotic distributional stability}

As explained in the introduction our objective is to prove asymptotic
distributional stability for the random walk adic transformation.
This is the goal of the present section. 

Let $G=\bbZ$ or $G=\bbR$ and let $\varphi:X\rightarrow\bbR$ be
aperiodic with $C_{\varphi,\alpha}<\infty$. As explained in the previous
section, in this case, the random walk adic transformation 
\[
\left(\hat{X}\times G,\left(\mathcal{B}\bigcap\hat{X}\right)\times\mathcal{B}\left(G\right),\mu,\tau_{\varphi}\right)
\]
is conservative and ergodic. Let $\chi$ be a standard Gaussian random
variable defined on some probability space. For two random variable
$Y$ and $Z$ we write $Y\overset{d}{=}Z$ if $Y$ has the same distribution
as $Z$. We prove
\begin{thm}
\label{thm:Main Theorem}The random walk adic transformation is disitributionally
stable with return sequence $a_{n}\propto\frac{n}{\sqrt{\log n}}$
and a random variable $Y\overset{d}{=}e^{-\chi^{2}}$ in the limit,
i.e 
\begin{equation}
\frac{1}{a_{n}}S_{n}\left(f\right)\overset{\mathcal{L}\left(m\right)}{\rightarrow}e^{-\chi^{2}}m\left(f\right)\label{eq: Distributional Stability-1}
\end{equation}
 for all $f\in L^{1}\left(m\right)$, $f\geq0$, where $S_{n}\left(f\right):=\sum_{k=0}^{n-1}f\circ\tau_{\varphi}^{k}$. \end{thm}
\begin{rem}
The theorem is valid for an aperiodic random walk on $G=\bbZ^{k}\times\bbR^{D-k}$
with return sequence $a_{n}\propto\frac{n}{\left(\log n\right)^{\frac{D}{2}}}$.
The changes needed for the proof in this setting are statements of
theorems in section (\ref{sec:Assumptions-on-the-observable}) for
$G=\bbZ^{k}\times\bbR^{D-k}$ as in \cite{AS}. In this case the random
variable $e^{-\chi_{D}^{2}}$ appears in the limit, where $\chi_{D}^{2}=\left\Vert \xi^{2}\right\Vert _{2}^{2}$
for $\xi$ a standard Gaussian random vector in $\bbR^{D}$. 
\end{rem}

\subsection{\label{sub:Overview-of-the proof}Overview of the proof. }

Similarly to the methods of \cite{AS} we split the $\tau$ orbit
up to time $n$ of a point $x\in\hat{X}$ into smaller blocks, where
each block is of the form $\left\{ T^{-l_{n}}\tau^{i}T^{l_{n}}x\right\} $
with $l_{n}\propto\log n$. Since the topological entropy of $T_{\beta}$
is $\log\beta$, each block is roughly of size $\beta^{l_{n}}$ (lemma
\ref{lem: Estimation for K(n,r)}). Over these blocks we are able
to estimate the sums $S_{n}\left(f\right)\left(x,y\right)$ for $f=\ind_{\hat{X}\times I}$
where $I$ is a bounded interval using the LLT (lemma \ref{lem:Estimation by LLT}).
This will allow us to prove that (\ref{eq: Distributional Stability-1})
holds for $f=\ind_{\hat{X}\times I}$, which by Hopf's ergodic theorem
is sufficient to obtain theorem \ref{thm:Main Theorem} (section \ref{sub:Proof-of-the main thm}).

\subsection{Conventions and notations.}

Throughout the remaining part of this paper we use the following conventions: 
\begin{enumerate}
\item For $a,b\in\bbR$, $c>0$ we write $a=b\pm c$ if $a\leq b+c$ and
$a\geq b-c$. 
\item $For$$I\subseteq G$, $\left|I\right|$ denotes the Haar measure
of $I$ (we use this in order to distinguish between the Lebesgue
measure on $\hat{X}$ which we denote by $\Leb$ and the Haar measure
on $G=\bbR$ or $G=\bbZ$.
\item $S_{n}\left(f\right)\left(x,y\right):=\sum_{k=0}^{n-1}f\left(\tau_{\varphi}^{k}\left(x,y\right)\right)$
\item $\phi_{n}:=\sum_{k=0}^{n-1}\phi\left(\tau^{k}x\right)$
\item $E_{m}\left(f\right):=\int_{X}f\left(x\right)\, dm\left(x\right)$;
$Var_{m}\left(f\right)=E_{m}\left(f^{2}\right)-E_{m}^{2}\left(f\right)$ 
\item For $x\sim_{\mathcal{T}\left(T_{\beta}\right)}x'$, set $N\left(x,x'\right):=\min\left\{ n\in\bbN:\, x_{j}=x_{j}^{'}\,\forall j\geq n\right\} $
\end{enumerate}

\subsection{Estimates.}

Set $J_{n}\left(x\right)=\#\left\{ y:\, y\in T^{-n}\left(x\right)\right\} $. 

For $n$ fixed, we call a point $x\in\hat{X}$ 
\begin{itemize}
\item $n$-minimal if $x=\min\left\{ T_{\beta}^{-n}\left(T_{\beta}^{n}x\right)\right\} $
and
\item $n$-maximal if $x=\max\left\{ T_{\beta}^{-n}\left(T_{\beta}^{n}x\right)\right\} $; 
\end{itemize}
where the minimum and the maximum are with respect to the reverse
lexicographic order. 

Define $K_{n}:\hat{X}\rightarrow\bbN$ and $\tau_{n}:\hat{X}\rightarrow\hat{X}$
by 
\[
K_{n}\left(x\right)=\min\left\{ k:\:\tau^{k}\left(x\right)\, is\, n\, maximal\right\} 
\]
and $\tau_{n}\left(x\right):\hat{X}\rightarrow\hat{X}$ by 
\[
\tau_{n}\left(x\right)=\tau^{K_{n}\left(x\right)+1}.
\]
Then 
\begin{itemize}
\item $\tau_{n}\left(x\right)$ is $n$-minimal as it must have zeroes in
the first $n$ coordinates of its $\beta$-expansion (see section
\ref{sub:Adic-transformation} for details on the structure of $\tau$).
\item $T^{n}\left(\tau_{n}x\right)=\tau\left(T^{n}x\right).$ 
\item $K_{n}\left(x\right)\leq J_{n}\left(x\right):=\#T^{-n}\left(T^{n}\left(x\right)\right)$
with equality if $x$ is $n$-minimal. 
\end{itemize}
Set $K_{n}^{r}\left(x\right):=K_{n}\left(x\right)+\sum_{j=1}^{r-1}K_{n}\left(\tau_{n}^{j}\left(x\right)\right)$
and $K_{n}^{0}=0$. 

Eventually, as explained in \ref{sub:Overview-of-the proof}, for
every $x\in\hat{X}$ we approximate $n$ by $K_{l_{n}}^{r}\left(x\right)$,
where $l_{n}\sim\log_{\beta}n$ and $r=r_{n}\left(x\right)$ is large.
This allows us to split the orbit of $x$ under $\tau$ up to time
$n$, into blocks of the form $T_{\beta}^{-l_{n}}\left(T_{\beta}^{l_{n}}\left(\tau_{n}^{j}x\right)\right)$
with cardinality of each block equal to $K_{l_{n}}\left(\tau_{l_{n}}^{j}\left(x\right)\right)$,
$j=1,...,r-1$. On each of these blocks, we are able to use the local
limit theorem, in order to obtain a total estimate for $\ind_{\hat{X}\times I}\left(S_{K_{l_{n}}^{r}}\left(x,0\right)\right)$
where $I=\left\{ 0\right\} $ if $G=\bbZ$ and $I$ is Riemann integrable
with $\left|I\right|<\infty$ if $G=\bbR$. 

We start with a lemma that provides an estimate of $K_{n}^{r}\left(x\right)$
on a large set of $x\in\hat{X}$. Note that the set depends on $n$,
but not on $r$, if $r$ is large enough. 
\begin{lem}
\label{lem: Estimation for K(n,r)}For every $\epsilon>0$ there exist
$R,N\in\bbN$ such that for every $n>N$, there exists a set $A_{n}^{R}$
with $\Leb\left(A_{n}^{R}\right)\geq1-\epsilon$, such that for every
$r>R$ and $x\in A_{n}^{R}$ , $K_{n}^{r}\left(x\right)=\beta^{n}r\left(1\pm\epsilon\right)$. \end{lem}
\begin{proof}
Fix $\epsilon>0$. We have 

\begin{eqnarray*}
K_{n}^{r}\left(x\right) & = & K_{n}\left(x\right)+\sum_{j=1}^{r-1}K_{n}\left(\tau_{n}^{j}x\right)\\
 & = & K_{n}\left(x\right)+\sum_{j=1}^{r-1}J_{n}\left(\tau_{n}^{j}x\right)\\
 & = & K_{n}\left(x\right)+\sum_{j=1}^{r-1}\#\left\{ T_{\beta}^{-n}\left(T_{\beta}^{n}\tau_{n}^{j}x\right)\right\} \\
 & = & K_{n}\left(x\right)+\sum_{j=1}^{r-1}\#\left\{ T_{\beta}^{-n}\tau^{j}T_{\beta}^{n}x\right\} 
\end{eqnarray*}
where the last equality follows form $ $$T^{n}\left(\tau_{n}x\right)=\tau\left(T^{n}x\right)$. 

By definition of the transfer operator $\mathcal{L}$ (see section
\ref{sec:Assumptions-on-the-observable}), 
\begin{eqnarray*}
\#\left\{ T^{-n}\tau^{j}T^{n}x\right\}  & = & \left(\mathcal{L}_{n}\ind\right)\left(\tau^{j}T_{\beta}^{n}x\right)\\
 & = & \beta^{n}h\left(\tau^{j}T_{\beta}^{n}x\right)\hat{T}_{\beta}^{n}\left(\frac{1}{h}\right)\left(\tau^{j}T_{\beta}^{n}x\right)\\
 & = & \beta^{n}h\left(\tau^{j}T_{\beta}^{n}x\right)\left(E_{m}\left(\frac{1}{h}\right)\pm\eta^{n}\right)\\
 & = & \beta^{n}\left(h\left(\tau^{j}T_{\beta}^{n}x\right)\pm\eta^{n}\right)
\end{eqnarray*}
where $0<\eta<1$. It follows that 
\[
K_{n}^{r}\left(x\right)=K_{n}\left(x\right)+\beta^{n}\left(\sum_{j=1}^{r-1}h\left(\tau^{j}T_{\beta}^{n}x\right)\pm r\eta^{n}\right).
\]
A similar computation gives $K_{n}\left(x\right)\leq\#T^{-n}\left(T^{n}x\right)\leq C\beta^{n}$
where $C$ is some constant. 

Set $A_{n}^{R}\left(\epsilon\right):=\left\{ x:\sum_{j=1}^{r-1}h\left(\tau^{j}T^{n}x\right)=r\pm\frac{\epsilon}{2}\:\forall r>R\right\} $
and consider the set $A_{1}^{R}\left(\epsilon\right)$. Since $E_{\lambda}\left(h\right)=1$,
by the ergodic theorem we have that $\Leb\left(A_{1}^{R}\left(\epsilon\right)\right)\geq1-\epsilon$
if $R$ is large enough. Since every measurable $A\subseteq X$ satisfies
$\left(1-\frac{1}{\beta}\right)\mathcal{\Leb}\left(A\right)\leq m\left(A\right)\leq\frac{1}{1-\frac{1}{\beta}}\mathcal{\Leb}\left(A\right)$,
it follows that there exists $R>1$ such that for every $n\in\bbN$,
$m\left(A_{n}^{R}\left(\epsilon\right)\right)=m\left(T^{-n}A_{1}^{R}\left(\epsilon\right)\right)=m\left(A_{1}^{R}\left(\epsilon\right)\right)\geq1-\epsilon$.
Therefore, there exists $R'$ such that for every $n\in\bbN$, $r>R'$,
$\Leb\left(A_{n}^{r}\left(\epsilon\right)\right)\geq1-\epsilon$.
Let $R'',\, N\in\bbN$ be such that $\forall n>N$, $r>R''$, we have
$\eta^{n}<\frac{\epsilon}{4}$ and $\frac{C}{r}<\frac{\epsilon}{4}$.
Then for $n>N$, $r>\max\left(R',R''\right)$, $x\in A_{n}^{R}:=A_{n}^{\max\left(R',R''\right)}\left(\epsilon\right)$
we have 
\begin{eqnarray*}
K_{n}^{r}\left(x\right) & \leq & K_{n}\left(x\right)+\beta^{n}\left(\sum_{j=1}^{r-1}h\left(\tau^{j}T_{\beta}^{n}x\right)+r\eta^{n}\right)\\
 & \leq & \beta^{n}r\left(1+\frac{C}{r}+\eta^{n}+\frac{\epsilon}{2}\right)\\
 & < & \beta^{n}r\left(1+\epsilon\right)
\end{eqnarray*}
and similarly\\*
\begin{eqnarray*}
K_{n}^{r}\left(x\right) & \geq & \beta^{n}\left(\sum_{j=1}^{r-1}h\left(\tau^{j}T_{\beta}^{n}x\right)-r\eta^{n}\right)\\
 & > & \beta^{n}r\left(1-\epsilon\right).
\end{eqnarray*}
The result follows from this.\end{proof}
\begin{lem}
\label{lem: Calculation of S_K(n,r)}For $I\subseteq G$ measurable,
\\* 

\begin{align*}
S_{K_{n}^{r}\left(x\right)} & \left(\ind_{\hat{X}\times I}\right)\left(x,y\right)\\
 & \leq\sum_{j=0}^{r-1}\#\left\{ z\in T_{\beta}^{-n}\left(\tau^{j}T_{\beta}^{n}\left(x\right)\right):\,\sum_{k=0}^{n+N\left(T^{n}x,\tau^{j}T^{n}x\right)}\varphi\left(T_{\beta}^{k}x\right)-\varphi\left(T_{\beta}^{k}\left(z\right)\right)\in I-y\right\} 
\end{align*}
and 

\begin{align*}
S_{K_{n}^{r}\left(x\right)} & \left(\ind_{\hat{X}\times I}\right)\left(x,y\right)\\
 & \geq\sum_{j=1}^{r-1}\#\left\{ z\in T_{\beta}^{-n}\left(\tau^{j}T_{\beta}^{n}\left(x\right)\right):\,\sum_{k=0}^{n+N\left(T^{n}x,\tau^{j}T^{n}x\right)}\varphi\left(T_{\beta}^{k}x\right)-\varphi\left(T_{\beta}^{k}\left(z\right)\right)\in I-y\right\} .
\end{align*}
\end{lem}
\begin{proof}
By definition
\begin{align}
S_{K_{n}^{r}\left(x\right)} & \left(\ind_{\hat{X}\times I}\right)\left(x,0\right)\nonumber \\
 & =S_{K_{n}\left(x\right)}\left(\ind_{\hat{X}\times I}\right)\left(x,y\right)+\sum_{j=1}^{r-1}S_{K_{n}^{j+1}\left(x\right)}\left(\ind_{\hat{X}\times I}\right)\left(x,y\right)-S_{K_{n}^{j}\left(x\right)}\left(\ind_{\hat{X}\times I}\right)\left(x,y\right).\label{eq: S_K(n,r) equation}
\end{align}
For fixed $j\geq1$,

\begin{align*}
S_{K_{n}^{j+1}\left(x\right)}\left(\ind_{\hat{X}\times I}\right)\left(x,y\right) & -S_{K_{n}^{j}\left(x\right)}\left(\ind_{\hat{X}\times I}\right)\left(x,y\right)\\
 & =\sum_{l=K_{n}^{j}\left(x\right)}^{K_{n}^{j+1}\left(x\right)-1}\ind_{\hat{X}\times I}\left(\tau^{l}x,\, y+\phi_{l}\left(x\right)\right)\\
 & =\sum_{l=0}^{K_{n}\left(\tau_{n}^{j}\left(x\right)\right)-1}\ind_{\hat{X}\times I}\left(\tau^{l}\left(\tau^{K_{n}^{j}\left(x\right)}x\right),\, y+\phi_{K_{n}^{\left(j\right)}\left(x\right)+l}\left(x\right)\right)\\
 & =\sum_{l=0}^{J_{n}\left(\tau_{n}^{j}\left(x\right)\right)-1}\ind_{\hat{X}\times I}\left(\tau^{l}\left(\tau^{K_{n}^{j}\left(x\right)}x\right),\, y+\phi_{K_{n}^{\left(j\right)}\left(x\right)+l}\left(x\right)\right).
\end{align*}
Now by the properties listed in the beginning of this section, 

\[
\left\{ \tau^{l}\left(\tau^{K_{n}^{j}\left(x\right)}\left(x\right)\right):\, l=0,...,J_{n}\left(\tau_{n}^{j}\left(x\right)\right)-1\right\} =T_{\beta}^{-n}\left(T_{\beta}^{n}\left(\tau^{K_{n}^{j}\left(x\right)}x\right)\right)=T_{\beta}^{-n}\left(\tau^{j}\left(T_{\beta}^{n}x\right)\right).
\]
Moreover, since for $x\sim_{\mathcal{T}\left(T_{\beta}\right)}x'$,
$N\left(x,x'\right)=\min\left\{ n\in\bbN:\, x_{j}=x_{j}^{'}\,\forall j\geq n\right\} $,
for $M>0$, we have

\begin{align*}
\phi_{M}\left(x\right) & =\sum_{i=0}^{M-1}\phi\left(\tau^{i}x\right)\\
 & =\sum_{i=0}^{M-1}\psi\left(\tau^{i}x,\tau^{i+1}x\right)\\
 & =\sum_{i=0}^{M-1}\sum_{k=0}^{\infty}\varphi\left(T_{\beta}^{k}\left(\tau^{i}x\right)\right)-\varphi\left(T_{\beta}^{k}\left(\tau^{i+1}x\right)\right)\\
 & =\sum_{i=0}^{M-1}\sum_{k=0}^{N\left(\tau^{i}x,\tau^{i+1}x\right)}\varphi\left(T_{\beta}^{k}\left(\tau^{i}x\right)\right)-\varphi\left(T_{\beta}^{k}\left(\tau^{i+1}x\right)\right)\\
 & =\sum_{i=0}^{M-1}\sum_{k=0}^{N\left(x,\tau^{M}x\right)}\varphi\left(T_{\beta}^{k}\left(\tau^{i}x\right)\right)-\varphi\left(T_{\beta}^{k}\left(\tau^{i+1}x\right)\right)\\
 & =\sum_{k=0}^{N\left(x,\tau^{M}x\right)}\varphi\left(T_{\beta}^{k}x\right)-\varphi\left(T_{\beta}^{k}\left(\tau^{M}x\right)\right)
\end{align*}
where the one prior to the last equality follows because $N\left(\tau^{i}x,\tau^{i+1}x\right)\leq N\left(x,\tau^{M}x\right)$
for $i\leq M-1$ and the extra terms in the sum vanish. It follows
that 
\begin{align*}
\sum_{l=0}^{J_{n}\left(\tau_{n}^{\left(j\right)}\left(x\right)\right)-1}\ind_{\left\{ \hat{X}\times I\right\} } & \left(\tau^{l}\left(\tau^{K_{n}^{\left(j\right)}\left(x\right)}x\right),\, y+\phi_{K_{n}^{\left(j\right)}\left(x\right)+l}\left(x\right)\right)\\
 & =\#\left\{ z\in T_{\beta}^{-n}\left(\tau^{j}T_{\beta}^{n}\left(x\right)\right):\,\sum_{k=0}^{N\left(x,z\right)}\varphi\left(T_{\beta}^{k}x\right)-\varphi\left(T_{\beta}^{k}\left(z\right)\right)\in I-y\right\} .
\end{align*}
Similarly, since $K_{n}\left(x\right)\leq J_{n}\left(x\right)$, 
\begin{eqnarray*}
S_{K_{n}\left(x\right)}\left(\ind_{\hat{X}\times I}\right)\left(x,y\right) & = & \sum_{l=0}^{K_{n}\left(x\right)-1}\ind_{\hat{X}\times I}\left(\tau^{l}\left(x\right),y+\phi_{l}\left(x\right)\right)\\
 & \leq & \sum_{l=0}^{J_{n}\left(x\right)}\ind_{\hat{X}\times I}\left(\tau^{l}\left(x\right),y+\phi_{l}\left(x\right)\right)\\
 & = & \#\left\{ z\in T_{\beta}^{-n}\left(T_{\beta}^{n}x\right):\,\sum_{k=0}^{N\left(x,z\right)}\varphi\left(T_{\beta}^{k}x\right)-\varphi\left(T_{\beta}^{k}\left(z\right)\right)\in I-y\right\} \\
 & = & \#\left\{ z\in T_{\beta}^{-n}\left(T_{\beta}^{n}\right):\,\sum_{k=0}^{n}\varphi\left(T_{\beta}^{k}x\right)-\varphi\left(T_{\beta}^{k}\left(x\right)\right)\in I-y\right\} ,
\end{eqnarray*}
where the last equality follows since for $z\in T_{\beta}^{-n}\left(T_{\beta}^{n}x\right)$,
$N\left(x,z\right)\leq n$ and the extra terms in the sum vanish. 

Note that for $j\geq1$, $z\in T^{-n}\left(\tau^{j}T^{n}\left(x\right)\right)$,
\[
N\left(x,z\right)=n+N\left(T_{\beta}^{n}x,T_{\beta}^{n}z\right)=n+N\left(T_{\beta}^{n}x,\tau^{j}T_{\beta}^{n}x\right).
\]
Thus 
\begin{align*}
\#\Biggl\{ z\in T_{\beta}^{-n} & \left(\tau^{j}T_{\beta}^{n}\left(x\right)\right):\,\sum_{k=0}^{N\left(x,z\right)}\varphi\left(T_{\beta}^{k}x\right)-\varphi\left(T_{\beta}^{k}\left(z\right)\right)\in I-y\Biggr\}\\
 & =\#\left\{ z\in T_{\beta}^{-n}\left(\tau^{j}T_{\beta}^{n}\left(x\right)\right):\,\sum_{k=0}^{n+N\left(T^{n}x,\tau^{j}T^{n}x\right)}\varphi\left(T_{\beta}^{k}x\right)-\varphi\left(T_{\beta}^{k}\left(z\right)\right)\in I-y\right\} 
\end{align*}
whence the lemma is proved by summing over $j$ and dropping the term
$S_{K_{n}\left(x\right)}$ from the sum in (\ref{eq: S_K(n,r) equation})
for the lower bound. 
\end{proof}
The next lemma shows that $\max_{0\leq j\leq r}N\left(T_{\beta}^{n}x,\tau^{j}T_{\beta}^{n}x\right)$
is negligible compared to $n$, for all $r$ bounded by some constant.
This is used in lemma \ref{lem:Estimation by LLT} for estimating
sums of the type $\sum_{k=0}^{n+N\left(T^{n}x,\tau^{j}T^{n}x\right)}\varphi\left(T_{\beta}^{k}x\right)-\varphi\left(T_{\beta}^{k}\left(z\right)\right)$
using the LLT. 
\begin{lem}
\label{lem:Estimation by BC}Let $C>0$. Then for all $r<C$, and
$M>C\log\beta$ the set $D_{n}^{r}\left(M\right):=\left\{ x\in\hat{X}:\,\max_{1\leq j\leq r}N\left(T_{\beta}^{n}x,\tau^{j}T_{\beta}^{n}x\right)\geq M\log n\right\} $
has Lebesgue measure $0$ if $n$ is large enough. \end{lem}
\begin{proof}
Since the quantity $N\left(x,\tau^{j}x\right)$ increases as $j$
increases, we have
\begin{eqnarray*}
C_{n}\left(M\right) & \subseteq & \left\{ x:\, N\left(T_{\beta}^{n}x,\tau^{r}\left(T_{\beta}^{n}x\right)\right)\geq M\log n\right\} \\
 & \subseteq & \left\{ x:\,\max_{0\leq i\leq r-1}N\left(\tau^{i}T_{\beta}^{n}x,\tau^{i+1}T_{\beta}^{n}x\right)\geq\frac{M\log n}{r}\right\} .
\end{eqnarray*}
Since $\tau$ preserves the Lebesgue measure, 
\begin{eqnarray*}
\Leb\left\{ x:\,\max_{0\leq i\leq r}N\left(\tau^{i}x,\tau^{i+1}x\right)\geq\frac{M\log n}{r}\right\}  & = & \lambda\left\{ \bigcup_{i=0}^{r-1}\left\{ x:\, N\left(\tau^{i}x,\tau^{i+1}x\right)\geq\frac{M\log n}{r}\right\} \right\} \\
 & \leq & r\Leb\left\{ x:\, N\left(x,\tau x\right)\geq\frac{M\log n}{r}\right\} .
\end{eqnarray*}
Since $N\left(x,\tau x\right)\geq\frac{M\log n}{r}$ implies that
$T^{i}x\geq1-\frac{1}{\beta}$ for all $i\leq\frac{M\log n}{r}$ (see
proposition \ref{prop: Identification of adic transformation} and
the proof therein), we have $\lambda\left\{ x:\, N\left(x,\tau x\right)\geq\frac{M\log n}{r}\right\} \leq C'\frac{1}{\beta^{\frac{M\log n}{r}}}$
where $C'$ is some constant. Since the sum $\sum_{n=1}^{\infty}\beta^{\frac{-M\log n}{r}}$
converges if $\frac{M}{r}>\log\beta$, it follows by the Borel-Cantelli
lemma that $\Leb\left(D_{n}^{r}\left(M\right)\right)=0$ $ $if $n$
is large enough. \end{proof}
\begin{lem}
\label{lem:Estimation by LLT}Let $I=\left\{ 0\right\} $ if $G=\bbZ$
and $I\subseteq\bbR$ a bounded interval if $G=\bbR$ and let $C$,
$\delta$ be some positive constant,. Then for all $\epsilon>0$,
$r<C$ there exists $N$ such that for all $n>N$, $x\in\hat{X}$,
$y\in I$, 
\begin{align*}
\ind_{B\left(0,\delta\right)}\left(\frac{\bar{\varphi}_{n}\left(x\right)}{\sqrt{n}}\right) & S_{K_{n}^{r}\left(x\right)}\left(\ind_{\hat{X}\times I}\right)\left(x,y\right)\\
\leq & \ind_{B\left(0,\delta\right)}\left(\frac{\bar{\varphi}_{n}\left(x\right)}{\sqrt{n}}\right)\left(\frac{\left|I\right|\beta^{n}\sum_{j=0}^{r-1}h\left(\tau^{j}T_{\beta}^{n}x\right)}{\sigma\sqrt{2\pi n}}\left(e^{-\frac{\bar{\varphi}_{n}^{2}\left(x\right)}{2\sigma^{2}n}}+\epsilon\right)\right)
\end{align*}
and 

\begin{align*}
\ind_{B\left(0,\delta\right)}\left(\frac{\bar{\varphi}_{n}\left(x\right)}{\sqrt{n}}\right) & S_{K_{n}^{r}\left(x\right)}\left(\ind_{\hat{X}\times I}\right)\left(x,y\right)\\
\geq & \ind_{B\left(0,\delta\right)}\left(\frac{\bar{\varphi}_{n}\left(x\right)}{\sqrt{n}}\right)\left(\frac{\left|I\right|\beta^{n}\sum_{j=1}^{r-1}h\left(\tau^{j}T_{\beta}^{n}x\right)}{\sigma\sqrt{2\pi n}}\left(e^{-\frac{\bar{\varphi}_{n}^{2}\left(x\right)}{2\sigma^{2}n}}-\epsilon\right)\right)
\end{align*}
where $\sigma^{2}=\frac{1}{n}Var_{m}\left(\sum_{k=0}^{n}\varphi\circ T_{\beta}^{k}\right)$. \end{lem}
\begin{proof}
By lemma \ref{lem: Estimation for K(n,r)} 
\begin{align*}
S_{K_{n}^{r}\left(x\right)} & \left(\ind_{\hat{X}\times I}\right)\left(x,y\right)\\
 & \leq\sum_{j=0}^{r-1}\#\left\{ z\in T_{\beta}^{-n}\left(\tau^{j}T_{\beta}^{n}\left(x\right)\right):\,\sum_{k=0}^{n+N\left(T^{n}x,\tau^{j}T^{n}x\right)}\varphi\left(T_{\beta}^{k}x\right)-\varphi\left(T_{\beta}^{k}\left(z\right)\right)\in I-y\right\} 
\end{align*}
and
\begin{align*}
S_{K_{n}^{r}\left(x\right)} & \left(\ind_{\hat{X}\times I}\right)\left(x,y\right)\\
 & \geq\sum_{j=1}^{r}\#\left\{ z\in T_{\beta}^{-n}\left(\tau^{j}T_{\beta}^{n}\left(x\right)\right):\,\sum_{k=0}^{n+N\left(T^{n}x,\tau^{j}T^{n}x\right)}\varphi\left(T_{\beta}^{k}x\right)-\varphi\left(T_{\beta}^{k}\left(z\right)\right)\in I-y\right\} .
\end{align*}
 Now for fixed $j$, setting $k_{n}\left(x,z\right)=\sum_{k=n+1}^{N\left(T^{n}x,\tau^{j}T^{n}x\right)}\varphi\left(T_{\beta}^{k}x\right)-\varphi\left(T_{\beta}^{k}z\right)$
we have, 
\[
\begin{aligned}\#\Biggl\{ z\in T_{\beta}^{-n}\left(\tau^{j}T_{\beta}^{n}\left(x\right)\right): & \sum_{k=0}^{n+N\left(T^{n}x,\tau^{j}T^{n}x\right)}\varphi\left(T_{\beta}^{k}x\right)-\varphi\left(T_{\beta}^{k}\left(z\right)\right)\in I\Biggr\}\\
 & =\sum_{z\in T^{-n}\tau^{j}T^{n}x}\ind_{\left\{ \varphi_{n}\left(z\right)\in\varphi_{n}\left(x\right)+k_{n}\left(x,z\right)+I-y\right\} }\\
 & =\mathcal{L}^{n}\left(\ind_{\left\{ \varphi_{n}\left(\cdot\right)\in\varphi_{n}\left(x\right)+k_{n}\left(x,\cdot\right)+I-y\right\} }\right)\left(\tau^{j}T^{n}x\right)\\
 & =\beta^{n}h\left(\tau^{j}T^{n}x\right)\hat{T}^{n}\left(\frac{\ind_{\left\{ \varphi_{n}\left(\cdot\right)\in\varphi_{n}\left(x\right)-k_{n}\left(x,\cdot\right)+I-y\right\} }}{h\left(\cdot\right)}\right)\left(\tau^{j}T^{n}x\right).
\end{aligned}
\]
Since $r$ is bounded, by lemma \ref{lem:Estimation by BC} if $n$
is large enough $k_{n}\left(x,z\right)\leq M\sup\left|\varphi\right|\log n$,
where $M$ is constant, and therefore by LLT, there exists $N$, such
that for all $n>N$, $x\in\hat{X}$, $y\in I$ we have 
\begin{align*}
\ind_{B\left(0,\delta\right)}\left(\frac{\bar{\varphi}_{n}}{\sqrt{n}}\right)\beta^{n}h\left(\tau^{j}T^{n}x\right) & \hat{T}^{n}\left(\frac{\ind_{\left\{ \varphi_{n}\left(\cdot\right)\in\varphi_{n}\left(x\right)-k_{n}\left(x,\cdot\right)+I-y\right\} }}{h\left(\cdot\right)}\right)\left(\tau^{j}T^{n}x\right)\\
 & =\ind_{B\left(0,\delta\right)}\left(\frac{\bar{\varphi}_{n}}{\sqrt{n}}\right)\beta^{n}h\left(\tau^{j}T^{n}x\right)\frac{\left|I\right|}{\sigma\sqrt{2\pi n}}\left(e^{-\frac{\bar{\varphi}_{n}^{2}\left(x\right)}{2\sigma^{2}l_{n}}}\pm\epsilon\right).
\end{align*}
whence the lemma is proved by summing over $j$. \end{proof}
\begin{lem}
\label{lem:Final lemma}Let $I$ be as in lemma \ref{lem:Estimation by LLT}.
For every $\epsilon>0$, there exists $K\in\bbN$, such that for every
$n>K$ there exists a set $A_{n}$ with $\Leb\left(A_{n}\right)\geq1-\epsilon$,
so that for every $x\in A_{n}$, $y\in I$ 
\[
\ind_{B\left(0,\delta\right)}\left(\frac{\bar{\varphi}_{n}}{\sqrt{n}}\right)\frac{\sqrt{l_{n}}}{n}S_{n}\left(\ind_{\hat{X}\times I}\right)\left(x,y\right)=\left|I\right|\frac{1}{\sigma\sqrt{2\pi}}\left(e^{-\frac{\bar{\varphi}_{l_{n}}^{2}\left(x\right)}{2\sigma^{2}l_{n}}}\pm\epsilon\right)
\]
where $l_{n}\sim\log_{\beta}n$, $\delta>0$. \end{lem}
\begin{proof}
Fix $\epsilon>0$. Let $N$, $R$ be as in lemma \ref{lem: Estimation for K(n,r)}.
Set $l_{n}=\left[\log_{\beta}\left(\frac{n}{\left(R+2\right)\left(1+\epsilon\right)}\right)\right]$
. Let $n$ be large enough so that $l_{n}>N$. Then by lemma \ref{lem: Estimation for K(n,r)}
there exists a set $A_{n}$ with $\lambda\left(A_{n}\right)>1-\epsilon$,
such that for all $r>R$, and all $x\in A_{n}$, 
\[
K_{l_{n}}^{r}\left(x\right)=\beta^{l_{n}}r\left(1\pm\epsilon\right).
\]
For $x\in A_{n}$, let $r_{n}\left(x\right)$ be such that $K_{l_{n}}^{r_{n}\left(x\right)}\left(x\right)\leq n<K_{l_{n}}^{r_{n}\left(x\right)+1}$.
Since for $r\leq R$, 
\[
K_{l_{n}}^{r+1}\left(x\right)\leq K_{_{l_{n}}}^{R+1}\left(x\right)=\beta^{l_{n}}\left(R+1\right)\left(1\pm\epsilon\right)\leq\frac{n}{\left(R+2\right)}\left(R+1\right)<n
\]
it follows that $r_{n}\left(x\right)>R$. Moreover, since 
\[
n\geq K_{l_{n}}^{r_{n}\left(x\right)}\geq\beta^{l_{n}}r_{n}\left(x\right)\left(1-\epsilon\right)\geq\left(\frac{n}{\left(R+2\right)\left(1+\epsilon\right)}-\beta\right)\left(r_{n}\left(x\right)\right)\left(1-\epsilon\right)
\]
we have $r_{n}\left(x\right)\leq C$ where $C$ depends only on $R$,$\epsilon$,$\beta$.
By lemma \ref{lem:Estimation by LLT} there exists $N'$ such that
for all $n>N'$, $r\leq C$, $x\in\hat{X}$, $y\in I$ 
\[
\ind_{B\left(0,\delta\right)}\left(\frac{\bar{\varphi}_{n}\left(x\right)}{\sqrt{n}}\right)S_{K_{n}^{r}\left(x\right)}\left(\ind_{\hat{X}\times I}\right)\left(x,y\right)\leq\ind_{B\left(0,\delta\right)}\left(\frac{\bar{\varphi}_{n}\left(x\right)}{\sqrt{n}}\right)\left(\left|I\right|\frac{\beta^{n}\sum_{j=0}^{r-1}h\left(\tau^{j}T_{\beta}^{n}x\right)}{\sqrt{n}}\left(e^{-\frac{\bar{\varphi}_{n}\left(x\right)}{2n}}+\epsilon\right)\right)
\]
and 
\[
\ind_{B\left(0,\delta\right)}\left(\frac{\bar{\varphi}_{n}\left(x\right)}{\sqrt{n}}\right)S_{K_{n}^{r}\left(x\right)}\geq\ind_{B\left(0,\delta\right)}\left(\frac{\bar{\varphi}_{n}\left(x\right)}{\sqrt{n}}\right)\left(\left|I\right|\frac{\beta^{n}\sum_{j=1}^{r-1}h\left(\tau^{j}T_{\beta}^{n}x\right)}{\sqrt{n}}\left(e^{-\frac{\bar{\varphi}_{n}\left(x\right)}{2n}}-\epsilon\right)\right)
\]
It follows that for $n$ such that $l_{n}>\max\left(N,N'\right)$,
and for all $x\in A_{n}$, 
\begin{eqnarray*}
S_{n}\left(\ind_{\hat{X}\times I}\right)\left(x,y\right) & \geq & S_{K_{l_{n}}^{r_{n}\left(x\right)}\left(x\right)}\left(\ind_{\hat{X}\times I}\right)\left(x,y\right)\\
 & \geq & \frac{\beta^{l_{n}}\sum_{j=1}^{r_{n}\left(x\right)-1}h\left(\tau^{j}T^{n}x\right)}{\sigma\sqrt{2\pi l_{n}}}\left(e^{-\frac{\bar{\varphi}_{l_{n}}\left(x\right)}{2l_{n}}}-\epsilon\right)\\
 & \geq & \frac{\beta^{l_{n}}r_{n}\left(x\right)\left(1-\epsilon\right)}{\sigma\sqrt{2\pi l_{n}}}\left(e^{-\frac{\bar{\varphi}_{l_{n}\left(x\right)}}{2l_{n}}}-\epsilon\right)
\end{eqnarray*}
where in the last inequality we use $\sum_{j=1}^{r_{n}\left(x\right)-1}h\left(\tau^{j}T^{n}x\right)\geq r_{n}\left(x\right)\left(1-\epsilon\right)$
for $r>R$, which we may assume to be true by the proof of lemma \ref{lem: Estimation for K(n,r)}.
Similarly
\begin{eqnarray*}
S_{n}\left(\ind_{\hat{X}\times I}\right)\left(x,0\right) & \leq & S_{K_{l_{n}}^{r_{n}\left(x\right)+1}\left(x\right)}\left(\ind_{\hat{X}\times I}\right)\left(x,0\right)\\
 & \leq & \frac{\beta^{l_{n}}\sum_{j=0}^{r_{n}\left(x\right)}h\left(\tau^{j}T^{n}x\right)}{\sigma\sqrt{2\pi l_{n}}}\left(e^{-\frac{\bar{\varphi}_{l_{n}}\left(x\right)}{2l_{n}}}+\epsilon\right)\\
 & \leq & \frac{\beta^{l_{n}}r_{n}\left(x\right)\left(1+\epsilon\right)}{\sigma\sqrt{2\pi l_{n}}}\left(e^{-\frac{\bar{\varphi_{l_{n}}}\left(x\right)}{2l_{n}}}+\epsilon\right).
\end{eqnarray*}
Since $ $$K_{l_{n}}^{r_{n}\left(x\right)}\leq n\leq K_{l_{n}}^{r_{n}\left(x\right)+1}$
we have 
\[
n\geq K_{l_{n}}^{r_{n}\left(x\right)}\geq\beta^{l_{n}}\left(r_{n}\left(x\right)-\epsilon\right)
\]
and 
\[
n\leq K_{l_{n}}^{r_{n}\left(x\right)+1}\leq\beta^{l_{n}}\left(r_{n}\left(x\right)+1+\epsilon\right).
\]
 It follows that 
\[
\frac{\beta^{l_{n}}r_{n}\left(x\right)}{n}\leq1+\frac{\beta^{l_{n}}\epsilon}{n}\leq1+\epsilon
\]
and 
\[
\frac{\beta^{l_{n}}r_{n}\left(x\right)}{n}\geq1-\frac{1}{R+2}-\frac{1}{n}
\]
and we may assume that $\frac{1}{R+2}<\epsilon$ by enlarging $R$
if necessary. Thus, 
\[
\frac{\sqrt{l_{n}}}{n}S_{n}\left(\ind_{\hat{X}\times I}\right)\left(x,0\right)\leq\frac{\left(1+\epsilon\right)^{2}}{\sigma\sqrt{2\pi}}\left(e^{-\frac{\bar{\varphi_{l_{n}}}\left(x\right)}{2l_{n}}}+\epsilon\right)
\]
and 
\[
\frac{\sqrt{l_{n}}}{n}S_{n}\left(\ind_{\hat{X}\times I}\right)\left(x,0\right)\geq\frac{\left(1-\epsilon\right)^{2}}{\sigma\sqrt{2\pi}}\left(e^{-\frac{\bar{\varphi_{l_{n}}}\left(x\right)}{2l_{n}}}-\epsilon\right)
\]
and the lemma follows from this. 
\end{proof}
The following lemma will only be used in the proof of bounded rational
ergodicity of the random walk adic transformation (see section \ref{sec:Bounded-rational-ergodicity}). 
\begin{lem}
\label{lem:Bounded rat ergodicity }Let $I$ be as in lemma \ref{lem:Estimation by LLT}.
There exists $C>0$, such that for all $n,r\in\bbN$, $\left(x,y\right)\in\hat{X}\times G$,
\[
S_{K_{n}^{r}\left(x\right)}\left(\ind_{\hat{X}\times I}\right)\left(x,y\right)\leq Cr\frac{\beta^{n}}{\sqrt{n}}.
\]
\end{lem}
\begin{proof}
By lemma \ref{lem: Calculation of S_K(n,r)}
\begin{alignat*}{1}
S_{K_{n}^{r}\left(x\right)}\left(\ind_{\hat{X}\times I}\right) & \left(x,y\right)\\
 & \leq\sum_{j=0}^{r-1}\#\left\{ z\in T_{\beta}^{-n}\left(\tau^{j}T_{\beta}^{n}\left(x\right)\right):\,\sum_{k=0}^{n+N\left(T^{n}x,\tau^{j}T^{n}x\right)}\varphi\left(T_{\beta}^{k}x\right)-\varphi\left(T_{\beta}^{k}\left(z\right)\right)\in I-y\right\} 
\end{alignat*}
 Similarly to the calculation in the proof of lemma \ref{lem:Estimation by LLT},
setting 
\[
k_{n}\left(x,z\right)=\sum_{k=n+1}^{N\left(T^{n}x,\tau^{j}T^{n}x\right)}\varphi\left(T_{\beta}^{k}x\right)-\varphi\left(T_{\beta}^{k}z\right)
\]
 we have
\begin{align*}
\sum_{j=0}^{r-1}\#\Biggl\{ z\in T_{\beta}^{-n}\left(\tau^{j}T_{\beta}^{n}\left(x\right)\right): & \,\sum_{k=0}^{n+N\left(T^{n}x,\tau^{j}T^{n}x\right)}\varphi\left(T_{\beta}^{k}x\right)-\varphi\left(T_{\beta}^{k}\left(z\right)\right)\in I-y\Biggr\}\\
 & =\sum_{j=0}^{r-1}\mathcal{L}^{n}\left(\ind_{\left\{ \varphi_{n}\left(\cdot\right)\in\varphi_{n}\left(x\right)+k_{n}\left(x,\cdot\right)+I-y\right\} }\right)\left(\tau^{j}T^{n}x\right)\\
 & =\sum_{j=0}^{r-1}\beta^{n}h\left(\tau^{j}T^{n}x\right)\hat{T}^{n}\left(\frac{\ind_{\left\{ \varphi_{n}\left(\cdot\right)\in\varphi_{n}\left(x\right)-k_{n}\left(x,\cdot\right)+I-y\right\} }}{h\left(\cdot\right)}\right)\left(\tau^{j}T^{n}x\right)\\
 & \leq C\cdot r\frac{\beta^{n}}{\sqrt{n}}.
\end{align*}
where $C$ is some constant. The last inequality follows from theorems
\ref{thm:Upper bound LLT}, \ref{thm:Upper bound LLT - continuous}
using $k_{n}\left(x,z\right)\leq M\log n$ (see lemma \ref{lem:Estimation by BC}). 
\end{proof}

\subsection{\label{sub:Proof-of-the main thm}Proof of theorem \ref{thm:Main Theorem}.}

Let $g$ be bounded and continuous on $\left[0,\infty\right]$ and
let $f\in L^{1}\left(\lambda\times dy\right)$, $f\geq0$. Our objective
is to prove that for $a_{n}\propto\frac{n}{\sqrt{\log n}}$, 
\[
g\left(S_{n}f\right)dm\longrightarrow Eg\left(e^{-\frac{1}{2}\chi^{2}}\right)
\]
 where $\chi$ is a standard Gaussian random variable. 

Fix $\epsilon>0$. Since we have $\frac{1}{n}Var_{m}\left(\sum\varphi\circ T_{\beta}^{i}\right)\longrightarrow\sigma^{2}>0$,
it follows from Chebychev's inequality that if $\delta$ is large
enough $m\left(\frac{\bar{\varphi}_{n}}{\sqrt{n}}\in B\left(0,\delta\right)\right)>1-\epsilon$
for all $n\in\bbN$. By the fact that the invariant density $h$ is
bounded from above and is bounded away from zero, the same is true
with the Lebesgue measure $\lambda$ instead of $m$. Thus, for every
$n\in N$, there exist a set $B_{n}$ with $\lambda\left(B_{n}\right)>1-\epsilon$
such that $\frac{\bar{\varphi}_{n}\left(x\right)}{\sqrt{n}}\in B\left(0,\delta\right)$
for all $x\in B_{n}$. By lemma \ref{lem:Final lemma} there exists
$K$, $l_{n}\sim\log n$, such that for all $n>K$, there exists a
set $A_{n}$ with $\lambda\left(A_{n}\right)>1-\epsilon$, so that
\[
\ind_{B\left(0,\delta\right)}\frac{\sqrt{l_{n}}}{n}S_{n}\left(\ind_{\hat{X}\times I}\right)\left(x,y\right)=\ind_{B\left(0,\delta\right)}\frac{\left|I\right|}{\sigma\sqrt{2\pi}}\left(e^{-\frac{\bar{\varphi}_{l_{n}}^{2}\left(s\right)}{2\sigma l_{n}}}\pm\epsilon\right)
\]
for every $x\in A_{n}$, $y\in I$. By the uniform continuity of the
function $g$ this implies that on the set $A_{n}$, 
\[
g\left(\ind_{B\left(0,\delta\right)}\frac{\sqrt{l_{n}}}{n}S_{n}\left(\ind_{\hat{X}\times I}\right)\left(x,y\right)\right)=g\left(\ind_{B\left(0,\delta\right)}\frac{\left|I\right|}{\sigma\sqrt{2\pi}}e^{-\frac{\bar{\varphi}_{l_{n}}^{2}\left(s\right)}{2\sigma l_{n}}}\right)\pm\epsilon
\]
Since $\lambda\left(A_{n}\bigcap B_{n}\right)>1-2\epsilon$ it follows
that 
\[
\left|\int\limits _{\hat{X}\times I}g\left(\frac{\sqrt{l}_{n}}{n}S_{n}\left(\ind_{\hat{X}\times I}\right)\left(x,y\right)\right)dm-\int\limits _{\hat{X}\times I}g\left(\frac{\left|I\right|}{\sigma\sqrt{2\pi}}e^{-\frac{\bar{\varphi}_{l_{n}}^{2}\left(s\right)}{2\sigma l_{n}}}\right)dm\right|\leq\epsilon+4\epsilon\sup\left|g\right|.
\]
Now by the CLT, 
\[
\int\limits _{\hat{X}\times I}g\left(\frac{\left|I\right|}{\sigma\sqrt{2\pi}}e^{-\frac{\bar{\varphi}_{l_{n}}^{2}\left(s\right)}{2\sigma l_{n}}}\right)dm\longrightarrow\int\limits _{I}E\left(g\left(\frac{\left|I\right|}{\sigma\sqrt{2\pi}}e^{-\chi^{2}}\right)\right)dy=\left|I\right|E\left(g\left(\frac{\left|I\right|}{\sigma\sqrt{2\pi}}e^{-\chi^{2}}\right)\right).
\]
It follows that there exists a sequence $a_{n}\propto\frac{n}{\sqrt{\log n}}$,
such that for all bounded and continuous $g$ on $\left[0,\infty\right]$,
$p=\frac{\ind_{\hat{X}\times I}}{m\left(\hat{X}\times I\right)}\in L^{\infty}\left(m\right)$,
$h=\ind_{\hat{X}\times G}$,
\begin{equation}
\int\limits _{\hat{X}\times G}g\left(\frac{1}{a_{n}}S_{n}\left(h\right)\right)\cdot pdm\longrightarrow E\left(g\left(m\left(h\right)\cdot e^{-\chi^{2}}\right)\right).\label{eq: Conv. For certain fn}
\end{equation}
We claim that this implies 
\begin{equation}
\frac{1}{a_{n}}S_{n}\left(h\right)\overset{\mathcal{L}}{\longrightarrow}e^{-\chi^{2}}.\label{eq: Conv. in dist.}
\end{equation}
To see this, assume by contradiction that this is not the case. Then
by definition, there exists a probability measure $q\ll m$, a function
$f\in C\left[0,\infty\right]$, $\epsilon>0$ and a subsequence $n_{k}$
such that 
\begin{equation}
\left|\int\limits _{\hat{X}\times G}f\left(\frac{1}{a_{n_{k}}}S_{n_{k}}\left(h\right)\right)\cdot qdm-E\left(f\left(m\left(h\right)\cdot e^{-\chi^{2}}\right)\right)\right|>\epsilon\label{eq: Contradiction inequality}
\end{equation}
for all $k\in\bbN$. By corollary 3.6.2 in \cite{Aa1}, there exists
a further subsequence $m_{l}:=n_{k_{l}}$ and a random variable $Y$
on $\left[0,\infty\right]$, such that 
\[
\frac{1}{a_{m_{l}}}S_{m_{l}}\left(h\right)\overset{\mathcal{L}}{\longrightarrow}Y.
\]
It follows that for all $g\in C\left[0,\infty\right]$, and a probability
measure $q\ll m$, 
\[
\int\limits _{\hat{X}\times G}g\left(\frac{1}{a_{m_{l}}}S_{m_{l}}\left(h\right)\right)\cdot qdm\longrightarrow E\left(g\left(m\left(h\right)\cdot Y\right)\right).
\]
But (\ref{eq: Conv. For certain fn}) implies that for $p=\frac{\ind_{\hat{X}\times I}}{m\left(\hat{X}\times Y\right)}$,
$ $$g\in C\left[0,\infty\right]$,
\[
\int\limits _{\hat{X}\times G}g\left(\frac{1}{a_{m_{l}}}S_{m_{l}}\left(h\right)\right)\cdot pdm\longrightarrow E\left(g\left(m\left(h\right)e^{-\chi^{2}}\right)\right)
\]
whence $Y$ has the same distribution as $e^{-\chi^{2}}$. This contradicts
(\ref{eq: Contradiction inequality}) and therefore (\ref{eq: Conv. in dist.})
holds. 

As explained in the introduction, by Hopf's ergodic theorem (\cite[Corollary 3.6.2]{Aa1}),
(\ref{eq: Conv. in dist.}) implies $\frac{1}{a_{n}}S_{n}\overset{\nu}{\longrightarrow}Y$,
which proves the theorem.

\section{\label{sec:Bounded-rational-ergodicity}Bounded rational ergodicity}

In this section we prove:
\begin{thm}
The random walk adic transformation $\left(\hat{X}\times G,\left(\mathcal{B}\bigcap\hat{X}\right)\times\mathcal{B}\left(G\right),\mu,\tau_{\varphi}\right)$,
with $\varphi:X\rightarrow G$ satisfying the assumptions of theorem
\ref{thm:Main Theorem} is bounded rationally ergodic with return
sequence $a_{n}\propto\frac{\sqrt{\log n}}{n}$. \end{thm}
\begin{rem}
As in the case of asymptotical distributional stability, the theorem
is valid for an aperiodic random walk on $G=\bbZ^{k}\times\bbR^{D-k}$
with return sequence $a_{n}\propto\frac{n}{\left(\log n\right)^{\frac{D}{2}}}$.
The changes needed for the proof in this setting are statements of
theorems in section (\ref{sec:Assumptions-on-the-observable}) for
$G=\bbZ^{k}\times\bbR^{D-k}$ as in \cite{AS}.\end{rem}
\begin{proof}
Bounded rational ergodicity follows (see definition ) if we prove
that there exists a measurable $A\subseteq\hat{X}\times G$ and constants
$C,c>0$ such that 
\begin{equation}
\int_{A}S_{n}\left(\ind_{A}\right)\left(x,y\right)d\mu\geq\frac{cn}{\sqrt{\log n}}\label{eq: Rational ergodicity lower bound}
\end{equation}
and 
\begin{equation}
\left\Vert S_{n}\left(\ind_{A}\right)\right\Vert _{\infty}\leq\frac{Cn}{\sqrt{\log n}}.\label{eq:Rational ergodicity upper bound}
\end{equation}
Let $I=\left\{ 0\right\} $ in case $G=\bbZ$ and $I$ a bounded interval
in case $G=\bbR$. Fix $\epsilon>0$. As in the proof of theorem \ref{thm:Main Theorem}
since $\frac{1}{n}Var_{m}\left(\sum\varphi\circ T_{\beta}^{i}\right)\longrightarrow\sigma^{2}>0$,
it follows from Chebychev's inequality that if $\delta$ is large
enough then $m\left(\frac{\bar{\varphi}_{n}}{\sqrt{n}}\in B\left(0,\delta\right)\right)>1-\epsilon$
for all $n\in\bbN$. It follows that for every $n\in\bbN$, there
exist a set $B_{n}$ with $\lambda\left(B_{n}\right)>1-\epsilon$
such that $\frac{\bar{\varphi}_{n}\left(x\right)}{\sqrt{n}}\in B\left(0,\delta\right)$
for all $x\in B_{n}$. By lemma \ref{lem:Final lemma}, there exists
a set $A_{n}\subseteq\hat{X}$ with $\lambda\left(A_{n}\right)>1-\epsilon$,
and a sequence $l_{n}\sim\log n$ such that for every $x\in A_{n}$,
$y\in I$, 
\[
\ind_{B\left(0,\delta\right)}\left(\frac{\bar{\varphi}_{n}}{\sqrt{n}}\right)\frac{\sqrt{l_{n}}}{n}S_{n}\left(\ind_{\hat{X}\times I}\right)\left(x,y\right)=\left|I\right|\frac{1}{\sigma\sqrt{2\pi}}\left(e^{-\frac{\bar{\varphi}_{l_{n}}^{2}\left(x\right)}{2\sigma^{2}l_{n}}}\pm\epsilon\right).
\]
It immediately follows from this that there exists $c>0$ , such that
\[
\int\limits _{\hat{X}\times I}S_{n}\left(\ind_{\hat{X}\times I}\right)\left(x,y\right)d\mu\geq\int_{\hat{X}\times I}\ind_{B_{n}\bigcap A_{n}}\left(x\right)S_{n}\left(\ind\right)\left(x,y\right)d\mu\geq\frac{cn}{\sqrt{\log n}}
\]
whence (\ref{eq: Rational ergodicity lower bound}) is proved. 

To prove (\ref{eq:Rational ergodicity upper bound}) let $l_{n}:=\left[L\log_{\beta}n\right]$
for some constant $L$ to be chosen later and consider $K_{l_{n}}^{2}$.
As in the proof of lemma \ref{lem: Estimation for K(n,r)} we have
\[
K_{l_{n}}^{2}\left(x\right)=K_{l_{n}}\left(x\right)+\beta^{l_{n}}h\left(\tau T_{\beta}^{l_{n}}x\right)\pm\eta^{l_{n}}
\]
where $0<\eta<1$. Since $h\geq1-\frac{1}{\beta}$, 
\[
\beta^{l_{n}}h\left(\tau T_{\beta}^{l_{n}}x\right)\geq\frac{Ln}{\beta}\left(1-\frac{1}{\beta}\right)\pm\eta^{l_{n}}.
\]
It follows that there exists $L$ such that $K_{l_{n}}^{2}\geq n$
for all $n\in\bbN$. Thus, using lemma \ref{lem:Bounded rat ergodicity }

\begin{eqnarray*}
S_{n}\left(\ind_{\hat{X}\times I}\right)\left(x,y\right) & \leq & S_{K_{l_{n}^{2}\left(x\right)}}\left(\ind_{\hat{X}\times I}\right)\\
 & \leq & C\frac{\beta^{l_{n}}}{\sqrt{l_{n}}}\leq\tilde{C}\frac{n}{\sqrt{\log n}}
\end{eqnarray*}
whence \ref{eq:Rational ergodicity upper bound} follows. $ $\end{proof}

\end{document}